\documentclass[arxiv,reqno,twoside,a4paper,12pt]{amsart}

\usepackage[utf8]{inputenc}
\usepackage{amsmath}
\usepackage{amsfonts}
\usepackage{amssymb}
\usepackage{amsthm}
\usepackage{float}
\usepackage{mathtools}
\usepackage{csquotes}
\usepackage[numbers]{natbib}
\usepackage{bbold}

\usepackage[colorlinks=true,linkcolor=blue,citecolor=blue]{hyperref}

\usepackage[scaled]{helvet} 
\usepackage{courier} 
\usepackage[mathbf]{euler}
\usepackage[dvipsnames]{xcolor}

\usepackage[driver=pdftex,margin=3cm,heightrounded=true,centering]{geometry}

\usepackage{tikz-cd}
\usetikzlibrary{arrows,matrix,shapes,decorations.text, mindmap,shapes.misc}
\usepackage{metalogo}
\newcommand{\A}{\alpha}
\newcommand{\N}{\mathbb{N}}

\newcommand{\R}{\mathbb{R}}
\newcommand{\norm}[1]{\left\lVert#1\right\rVert}

\newcommand{\Vol}[0]{\text{Vol}}

\usepackage[numbers]{natbib}

\tikzset{cross/.style={cross out, draw=black, minimum size=2*(#1-\pgflinewidth), inner sep=0pt, outer sep=0pt},
cross/.default={1pt}}

\theoremstyle{plain}
\newtheorem{Thm}{Theorem}[section]

\newtheorem{Lem}[Thm]{Lemma}
\newtheorem{Cor}[Thm]{Corollary}
\newtheorem{Prop}[Thm]{Proposition}
\newtheorem{Def}[Thm]{Definition}
\newtheorem{Rem}[Thm]{Remark}

\theoremstyle{plain}
\newtheorem{thm}{theorem}
\newtheorem{Assump}[thm]{Assumption}

\numberwithin{equation}{section}

\title{}
\date{}
\author{}

\begin{document}

\title[Long-time existence of the singular Yamabe flow]
{Long-time existence of Yamabe flow on singular spaces with positive Yamabe constant}

\author{J\o rgen Olsen Lye}
\address{Mathematisches Institut,
Universit\"at Oldenburg,
26129 Oldenburg,
Germany}
\email{jorgen.olsen.lye@uni-oldenburg.de}

\author{Boris Vertman}
\address{Mathematisches Institut,
Universit\"at Oldenburg,
26129 Oldenburg,
Germany}
\email{boris.vertman@uni-oldenburg.de}

\subjclass[2000]{53C44; 58J35; 35K08}
\date{\today}

\begin{abstract}
{In this work we establish long-time existence of the normalized Yamabe flow 
with positive Yamabe constant on a class of manifolds that includes 
spaces with incomplete cone-edge singularities. We formulate our results 
axiomatically, so that our results extend to general stratified spaces as well, provided 
certain parabolic Schauder estimates hold. The central analytic tool is a
parabolic Moser iteration, which yields uniform upper and lower bounds on both the solution 
and the scalar curvature.}
\end{abstract}

\maketitle
\tableofcontents

\section{Introduction and statement of the main results}\label{intro-section}

The Yamabe conjecture states that for any compact, smooth
Riemannian manifold $(M,g_0)$ there exists a constant scalar curvature metric, conformal
to $g_0$. The first proof of this conjecture was initiated by Yamabe \cite{Yamabe} and continued by Trudinger \cite{Trudinger}
Aubin \cite{Aubin}, and Schoen \cite{Schoen}. The proof is based on the calculus of variations and elliptic partial differential 
equations. An alternative tool for proving the conjecture is due to Hamilton \cite{Hamilton}: the normalized Yamabe flow of a Riemannian manifold $(M,g_0)$,
which is a family $g\equiv g(t), t\in [0,T]$ of Riemannian metrics on $M$ such 
that the following evolution equation holds
 \begin{equation}\label{eq:YF}
 \partial_t g=-(S-\rho)g, \quad \rho := \Vol_g(M)^{-1} \int_M S\, d\Vol_g.
 \end{equation}
Here $S$ is the scalar curvature of $g$, $\Vol_g(M)$ the total volume of $M$ with respect to $g$ and $\rho$ is the average scalar curvature of $g$.
The normalization by $\rho$ ensures that the total volume does not change along the flow.
Hamilton \cite{Hamilton} introduced the Yamabe flow and also showed its long time existence. 
It preserves the conformal class of $g_0$ and ideally should converge
to a constant scalar curvature metric, thereby establishing the Yamabe conjecture by parabolic methods. \medskip

Establishing convergence of the normalized Yamabe flow is intricate already in the setting of smooth, compact manifolds.
In case of scalar negative, scalar flat and locally conformally flat scalar positive cases, convergence is due to Ye \cite{Ye}.  
The case of a non-conformally flat $g_0$ with positive scalar curvature is delicate and has been studied by Brendle 
\cite{Brendle, BrendleYF}. More specifically, \cite[p. 270]{Brendle},  \cite[p. 544]{BrendleYF} invoke the positive mass theorem, which is 
where the dimensional restriction in \cite{Brendle} and the spin assumption in \cite[Theorem 4]{BrendleYF} come from. 
Assuming \cite{SY17} to be correct, \cite{Brendle} and \cite{BrendleYF} cover all closed manifolds which are not conformally equivalent to spheres.
\medskip

In the non-compact setting, our understanding is limited. On complete manifolds, long-time existence 
has been discussed in various settings by Ma \cite{Ma}, Ma and An \cite{MaAn}, and the recent contribution by Schulz \cite{Schulz}.
On incomplete surfaces, where Ricci and Yamabe flows coincide, Giesen and Topping \cite{Topping, Topping2}
constructed a flow that becomes instantaneously complete. \medskip

In this work, we study the Yamabe flow on a general class of spaces that includes incomplete spaces with 
cone-edge (wedge) singularities or, more generally, stratified spaces with iterated cone-edge singularities. This continues a program
initiated in \cite{ShortTime, LongTime}, where 
existence and convergence of the Yamabe flow has been established in case of negative Yamabe invariant. 
Here, we study the positive case and, utilizing methods of Akutagawa, Carron, Mazzeo \cite{ACM},
we establish long time existence of the flow under certain mild geometric assumptions. We don't attempt 
to prove convergence here, in view of \cite{Brendle, BrendleYF} and the fact that we don't have a substitute 
for the positive mass theorem in the singular setting. Our main result is as follows.

\begin{Thm}\label{main-thm}
Let $(M,g_0)$ be a Riemannian manifold of dimension $n= \dim M \geq 3$, 
such that the following four assumptions (to be made precise below) hold:
\begin{enumerate}
\item The Yamabe constant $Y(M,g_0)$ is positive;
\item $(M,g_0)$ is admissible and in particular satisfies a Sobolev inequality;
\item Certain parabolic Schauder estimates hold on $(M,g_0)$;
\item The initial scalar curvature $S_0$ admits certain H\"older regularity.
\end{enumerate}
Under these assumptions, the normalized Yamabe flow of $g_0$ exists within the space of 
admissible spaces, with infinite existence time. 
\end{Thm}

\noindent Examples, where the assumptions of the theorem are satisfied, include spaces with incomplete wedge 
singularities. More general stratified spaces with iterated cone-edge metrics are also covered, provided
parabolic Schauder estimates continue to hold in that setting. \medskip

\noindent We now proceed with explaining the assumptions in detail.

\subsection{Normalized Yamabe flow and Yamabe constant}

Consider a Riemannian manifold $(M,g_0)$, with $g_0$ normalized such that the total volume $\Vol_{g_0}(M)=1$.
The Yamabe flow \eqref{eq:YF} preserves the conformal class
of the initial metric $g_0$ and, assuming $\dim M = n \geq 3$, we can write $g=u^{\frac{4}{n-2}}g_0$ 
for some function $u>0$ on $M_T = M \times [0,T]$ for some upper time limit $T>0$. Then the normalized Yamabe flow equation
can be equivalently written as an equation for $u$
\begin{equation}
\partial_t \left(u^{\frac{n+2}{n-2}}\right) = \frac{n+2}{4} \left(\rho u^{\frac{n+2}{n-2}}- L_0(u)\right),
\quad L_0 := S_0-\frac{4(n-1)}{n-2} \Delta_0,
\label{eq:YamabeFlow}
\end{equation}
where $L_0$ is the conformal Laplacian of $g_0$, defined in terms of the scalar curvature $S_0$
and the Laplace Beltrami operator $\Delta_0$ associated to the initial metric $g_0$. 
The scalar curvature $S$ of the evolving metric $g$ can be written $S=u^{-\frac{n+2}{n-2}} L_0(u)$, 
and the volume form of $g=u^{\frac{4}{n-2}}g_0$ is given by $d\Vol_g=u^{\frac{2n}{n-2}} d\mu$,
where we write $d\mu := d\Vol_{g_0}$ for the time-independent initial volume form. 
One computes 
\begin{align}\label{eq:DVol}
\partial_t d\Vol_g &=-\frac{n}{2} (S-\rho)\, d\Vol_g.
\end{align}
Hence, the total volume of $(M,g)$ is constant and thus equal to $1$ along the flow.
The average scalar curvature then takes the form
\begin{equation}
\rho=\int_M S\, d\Vol_g= \int_M L(u) u^{-\frac{n+2}{n-2}} u^{\frac{2n}{n-2}} \, d\mu= \int_{M}\frac{4(n-1)}{n-2}  \vert \nabla u\vert^2+S_0 u^2\, d\mu.
\label{eq:rho}
\end{equation} 
Explicit computations lead to the following evolution equation for the average scalar curvature
\begin{align} 
\label{eq:rhoEvol}
\partial_t \rho &=-\frac{n-2}{2} \int_M \left( \rho-S\right)^2 u^{\frac{2n}{n-2}}\, d\mu.
\end{align}
The latter evolution equation in particular implies that $\rho \equiv \rho(t)$ is non-increasing 
along the flow. We conclude the exposition with defining the Yamabe constant of $g_0$, which 
incidentally provides a lower bound for $\rho$.  Let $u$ be a solution of \eqref{eq:YamabeFlow}. We define the $L^q(M)$ spaces with respect to the integration measure $d\mu$. \medskip

\noindent We define the first Sobolev space $H^1(M)$ as the space of all 
$v: M \to \R$ such that first Sobolev norm, defined with respect to $d\mu$
and the pointwise norm associated to $g_0$
$$
\| v \|^2_{H^1(M)} := \int_M \nu^2 \, d\mu +\int_M \vert \nabla \nu\vert^2\, d\mu < \infty.
$$
Similary, we define $H^1(M,g)$ by using $d\Vol_g$ instead of $d\mu$, and the 
pointwise norm associated to $g$. If $u$ and $u^{-1}$ are both bounded, one easily checks $H^1(M) = H^1(M,g)$.
\medskip

\noindent We define the Yamabe invariant of $g_0$ as follows
\begin{equation}\begin{split}
Y(M,g_0) &:= \inf_{v\in H^1(M)\setminus \{0\}}  
\frac{\int_{M} \frac{4(n-1)}{n-2} \vert \nabla v\vert^2+S_0 v^2\, d\mu}{\norm{v}^2_{L^{\frac{2n}{n-2}}(M)}}
\\ &\leq \int_{M}\frac{4(n-1)}{n-2}  \vert \nabla u\vert^2+S_0 u^2\, d\mu \stackrel{\eqref{eq:rho}}{=} \rho,
\label{eq:YamabeConst}
\end{split}\end{equation}
where in the inequality we have used that for any solution $u$ of \eqref{eq:YamabeFlow}, 
$\norm{u}_{L^{\frac{2n}{n-2}}(M)} = d\Vol_g(M) \equiv 1$.
How one proceeds will depend heavily on the sign of the Yamabe constant. In this paper we will assume $Y(M,g_0)>0$. 
In particular, the average curvature $\rho$ is then positive and uniformly bounded away from zero along the normalized Yamabe flow.
\smallskip

\begin{Assump}\label{Y-assump}
The Yamabe constant $Y(M,g_0)$ is positive.
\end{Assump} \ \\[-8mm]

\subsection{A Sobolev inequality and other admissibility assumptions}
The Moser iteration arguments in this paper are strongly motivated by the related work 
of Akutagawa, Carron and Mazzeo \cite{ACM} on the Yamabe problem on stratified spaces.
Thus, similar to \cite{ACM}, we impose certain admissibility assumptions, which are 
naturally satisfied by certain compact stratified spaces with iterated cone-edge metrics. 

\begin{Def}\label{admissible}
Let $(M,g_0)$ be a smooth Riemannian manifold of dimension $n$. We call 
$(M,g_0)$ admissible, if it satisfies the following conditions.
\begin{itemize}
\item $(M,g_0)$ with volume form $d\mu = d\Vol_{g_0}$ has finite volume $\Vol_{g_0}(M)<\infty$. 
\item For any $\varepsilon>0$, there exist finitely many open balls $B_{2R_i}(x_i)\subset M$ such that 
\begin{align}
\Vol_{g_0} \left(M \backslash \bigcup_i B_{R_i}(x_i)\right) \leq \varepsilon.
\label{eq:Exhaustion}
\end{align}
\item Smooth, compactly supported functions $C^\infty_c(M)$ are dense in 
$H^1(M)$\footnote{This can be phrased as $H^1_0(M)=H^1(M)$.
Note that this rules out $M$ being the interior of a manifold with a codimension 1 boundary.}.
\item $(M,g_0)$ admits a Sobolev inequality of the following kind. Defining $L^q(M)$ spaces with respect to $d\mu$, 
there exist $A_0,B_0>0$ such that for all $f\in H^1(M)$
\begin{equation}
\norm{f}_{L^{\frac{2n}{n-2}}(M)}^2 \leq A_0 \norm{\nabla f}_{L^2(M)}^2 +B_0 \norm{f}_{L^2(M)}^2.
\label{eq:Sobolev}
\end{equation} 
\end{itemize}
\end{Def}

The main examples we have in mind are closed manifolds\footnote{This includes finite volume, complete manifolds, since (see \cite[Lemma 3.2, pp. 18-19]{Sobolev}, \cite[Remark 2), pp. 56-57]{Sobolev}) any finite volume, complete manifold satisfying the Sobolev inequality is compact.}  and regular parts of smoothly stratified spaces,
endowed with iterated cone-edge metrics. 
See \cite[Section 2.1]{ACM} for a definition of the latter. That the Sobolev inequality holds 
in this case is shown in \cite[Proposition 2.2]{ACM}. Note that the list of admissibility assumptions does not contain compactness. 
Nor do we specify explicitly how the metric $g_0$ looks near the singular strata of $\overline{M}$, in case of stratified spaces.
Restrictions on the local behaviour of the metric will instead be coded in $L^q$-data, like requiring the initial scalar curvature $S_0$ to be in 
$L^q(M)$ for suitable $q>0$. These requirements are stated in the theorems below, and will vary from statement to statement.
\smallskip

\begin{Assump}\label{admissible-assump}
$(M,g_0)$ is an admissible Riemannian manifold.
\end{Assump} \ \\[-8mm]

\noindent In what follows we want to relate the assumption of the Sobolev inequality 
\eqref{eq:Sobolev} in Definition \ref{admissible} to positivity of the Yamabe constant $Y(M,g_0)$.

\begin{Prop}\label{SY}
Assume $S_0\in L^{\infty}(M)$ and $Y(M,g_0)>0$. 
Then \eqref{eq:Sobolev} holds. 
\end{Prop}

\begin{proof}
Indeed, it follows directly from the definition of 
$Y(M,g_0)$ in \eqref{eq:YamabeConst} that
\[
\norm{f}^2_{L^{\frac{2n}{n-2}}(M)}\leq \frac{1}{Y(M,g_0)} \left( \frac{4(n-1)}{n-2}
\norm{\nabla f}^2_{L^2(M)}+\norm{S_0}_{L^\infty(M)} \norm{f}^2_{L^2(M)}\right)
\]
holds for all $f\in H^1(M)$. This is indeed the Sobolev inequality \eqref{eq:Sobolev}.
\end{proof}

\subsection{Parabolic Schauder estimates and short-time existence}\label{short-section}

Our proof requires intricate arguments involving the heat operator and its mapping 
properties, as is seen in the previous work by the second author jointly with Bahuaud 
\cite{ShortTime, LongTime} in the setting of spaces with incomplete wedge singularities. Here, we 
axiomatize these arguments into a definition of certain parabolic Schauder esimates,
having in mind further generalizations to stratified spaces. 

\begin{Def}\label{Schauder-def} $(M,g_0)$ satisfies parabolic Schauder estimates, if 
there is a sequence of Banach spaces $\{C^{k,\A} \equiv C^{k,\A}(M\times [0,T])\}_{\, k\in \N_0}$
of continuous functions on $M\times [0,T]$, for some $\A \in (0,1)$ and any $T>0$, with the following 
properties.

\begin{enumerate}
\item \textbf{Algebraic properties of the Banach spaces:} \medskip

\begin{enumerate}
\item For any $k\in \N_0$, the constant function $1 \in C^{k,\A}(M\times [0,T])$.
\medskip 

\item For any $k\in \N_0$ and any $u \in C^{k,\A}(M\times [0,T])$ uniformly bounded 
away from zero, we have for the inverse $u^{-1} \in C^{k,\A}(M\times [0,T])$. \medskip

\item For any $k\in \N_0$ we have $C^{k+1,\A}(M\times [0,T]) \subseteq C^{k,\A}(M\times [0,T])$.
\medskip

\item For any $k\geq 2$ and $\ell \leq k$ we have $C^{k,\A}\cdot C^{\ell,\A} \subseteq C^{\ell,\A}$.
Writing $\| \cdot \|_{\ell,\A}$ for the norm on $C^{\ell,\A}$ we have a uniform 
constant $C_{\ell,\A}$, such that for any $u \in C^{k,\A}$ and $v \in C^{\ell,\A}$
\begin{equation}\label{quadratic} \begin{split}
&\| u \cdot v \|_{\ell,\A} \leq C_{\ell,\A} \| u \|_{k,\A} \| v \|_{\ell,\A},  \\
\end{split} \end{equation}
\end{enumerate}

\item \textbf{Regularity properties of the Banach spaces:} \medskip

\begin{enumerate}

\item We have the following inclusions\footnote{The space $C^{1,\A}$ introduced 
in \cite{ShortTime} does not specify regularity under time differentiation. However, one can extend its definition
such that $\partial_{\sqrt{t}} u \equiv 2 \sqrt{t} \partial_t u \in C^{0,\A}$
for any $u\in C^{1,\A}$. This does not affect the other arguments in \cite{ShortTime}
and \cite{LongTime}.} 
\begin{equation}\label{LH}
\begin{split}
&C^{0,\A}(M\times [0,T])\subseteq C^0([0,T],L^2(M)), \\
&C^{1,\A}(M\times [0,T])\subseteq C^0([0,T],H^1(M)) \cap C^1((0,T),H^1(M)), \\
&C^{2,\A}(M\times [0,T])\subseteq C^1((0,T),H^1(M)) \cap L^\infty(M \times [0,T]).
\end{split}\end{equation}
Moreover, for any $u \in C^{0,\A}(M\times [0,T])$ and any fixed $p\in M$, 
the evaluation $u(p,\cdot)$ still lies in $C^{0,\A}$. The map $M\ni p \mapsto \| u(p,\cdot) \|_{0,\A}$ 
is again $L^2(M)$. \medskip

\item If $C^{k,\A}([0,T]) \subset C^{k,\A}(M\times [0,T])$
consists of functions that are constant on $M$, then the spaces $C^{2k,\A}([0,T])$ are characterized as follows
\begin{equation}\label{CT}
C^{2k,\A}([0,T]) = \{u \in C^{0,\A}([0,T]) \mid \partial_t^{k} u \in C^{0,\A}([0,T])\}.
\end{equation}

\item For any $k \in \N_0$ the following maps are bounded
\begin{equation}\label{DT}
\begin{split}
&\partial_t, \Delta_0: C^{k+2,\A}(M\times [0,T]) \to C^{k,\A}(M\times [0,T]), \\
&\nabla: C^{k+1,\A}(M\times [0,T]) \to C^{k,\A}(M\times [0,T]).
\end{split}
\end{equation}
 

\end{enumerate}

\item \textbf{Weak maximum principle for elements of the Banach spaces:} \medskip

\begin{enumerate}
\item Any $u \in C^{2,\A}(M\times [0,T])$ satisfies a weak maximum principle, that is for any
Cauchy sequence $\{q_\ell\}_{\ell \in \N} \subset M$ we have the following 
\begin{equation}\label{weak-max}
 \inf_M u = \lim_{\ell\to \infty} u(q_\ell) 
\Rightarrow \lim_{\ell\to \infty} (\Delta_0 u)(q_\ell) \geq 0. 
\end{equation}
In case the Cauchy sequence $\{q_\ell\}_{\ell \in \N}$ converges to an interior point $p \in M$,
where $u$ attains a global minimum, we have $\Delta_0 u(p)\geq 0$. \medskip
\end{enumerate}

\item \textbf{Mapping properties of the heat operator:} \medskip

\begin{enumerate}
\item The heat operator $e^{t\Delta_0}$ admits following mapping properties
\begin{equation}\label{mapping-heat}
\begin{split}
&e^{t\Delta_0} : C^{k,\A}(M\times [0,T]) \to C^{k+2,\A}(M\times [0,T]), \\
&e^{t\Delta_0} : C^{k,\A}(M\times [0,T]) \to t^\A C^{k+1,\A}(M\times [0,T]), \\
&e^{t\Delta_0} : L^\infty(M\times [0,T]) \to C^{1,\A}(M\times [0,T]).
\end{split}
\end{equation}
\end{enumerate}
If $e^{t\Delta_0}$ acts without convolution in time, then we have a bounded map
\begin{equation}\label{mapping-heat2}
e^{t\Delta_0}: C^{k,\A}(M) \to C^{k,\A}(M\times [0,T]).
\end{equation}

\item \textbf{Mapping properties of other solution operators:} \medskip

\begin{enumerate}

\item For any positive $a\in C^{1,\A}(M\times [0,T])$, uniformly bounded away from zero, there is 
a solution operator $Q$ for $(\partial_t - a \cdot \Delta_0) u = f, u(0)=0$, 
such that 
\begin{equation}\label{Q}
Q : C^{0,\A}(M\times [0,T]) \to C^{2,\A}(M\times [0,T]).
\end{equation}
If $a\in C^{2,\A}$, then additionally, $Q : C^{1,\A} \to C^{3,\A}$ is bounded.
\medskip

\item For any positive $a\in C^{1,\A}(M\times [0,T])$, uniformly bounded away from zero, there is 
a solution operator $R$ for $(\partial_t - a \cdot \Delta_0) u = 0, u(0)=f$, 
such that 
\begin{equation}\label{R}
R : C^{2,\A}(M) \to C^{2,\A}(M\times [0,T]),
\end{equation}
where $C^{k,\A}(M)$ denotes the subspace of $C^{k,\A}(M\times [0,T])$
consisting of time-independent functions. If $a\in C^{2,\A}$, then additionally, 
$R : C^{3,\A}(M) \to C^{3,\A}(M\times [0,T])$ is bounded.
\end{enumerate}
\end{enumerate}

\end{Def}

Clearly, parabolic Schauder estimates hold on smooth compact Riemannian manifolds. 
By \cite{ShortTime, LongTime} a manifold with a wedge singularity satisfies the 
parabolic Schauder estimates\footnote{In fact in the mapping properties of solution operators $Q$ and
$R$ we require here less than in \cite{LongTime}: in case $a\in C^{2,\A}$ we only ask for 
$Q : C^{1,\A} \to C^{3,\A}$ and $R : C^{3,\A}(M) \to C^{3,\A}$, while in \cite{LongTime} these
additional mapping properties are proved for one order higher.}, assuming that the wedge metric is feasible in the
sense of \cite[Definition 2.2]{LongTime}. The proof is based on the microlocal heat kernel 
description in \cite{MazVer}. Note that the choice of Banach spaces is not canonical, and one can e.g. 
use the scale of weighted H\"older spaces as in \cite{Ver-Ricci} instead. In view of the recent work by Albin and Gell-Redman 
\cite{AGR17}, we expect same parabolic Schauder estimates to hold on general 
stratified spaces with iterated cone-wedge metrics.

\begin{Assump}\label{Schauder-assump}
$(M,g_0)$ satisfies parabolic Schauder estimates.
\end{Assump} \ \\[-8mm]

\noindent Using parabolic Schauder estimates, we can prove short time existence and regularity of the 
renormalized Yamabe flow, exactly as in \cite[Theorem 1.7 and 4.1]{ShortTime} and by a slight adaptation of 
\cite[Proposition 4.8]{LongTime}.

\begin{Thm}\label{short}
Let $(M,g_0)$ satisfy parabolic Schauder estimates. 
Assume moreover that the scalar curvature $S_0$ of $g_0$ lies in $C^{1,\A}(M)$. 
Then the following holds.
\begin{enumerate}
\item The Yamabe flow 
\eqref{eq:YamabeFlow} admits for some $T>0$ 
sufficiently small, a solution 
$$
u \in C^{2,\A}(M\times [0,T]) \subseteq C^1((0,T),H^1(M)) \cap L^\infty(M \times [0,T]),
$$
that is positive and uniformly bounded away from zero\footnote{Later on, we will prove uniform lower 
bounds on $u$ for any finite $T>0$.}. \medskip

\item If a solution $u \in C^{2,\A}(M\times [0,T])$ to the Yamabe flow \eqref{eq:YamabeFlow} 
exists for a given $T>0$  and is uniformly bounded away from zero, then in fact 
$u \in C^{3,\A}(M\times [0,T])$. In particular we obtain
$$
S \in C^{1,\A}(M\times [0,T]) \subseteq C^0([0,T],H^1(M)) \cap C^1((0,T),H^1(M)).
$$
\end{enumerate}
\end{Thm}

\begin{proof} We shall only provide a brief proof outline.
The first statement is proved by setting up a fixed point argument in 
the Banach space $C^{2,\A}(M\times [0,T])$. If $u= 1+v \in C^{2,\A}(M\times [0,T])$
is a solution to \eqref{eq:YamabeFlow}, then $v$ satisfies the equation
\begin{equation}\label{lin}
\partial_t v -(n-1) \Delta_0 v = - \frac{n-2}{4} S_0 + \Phi (v),
\end{equation}
where $\Phi: C^{2,\A}(M\times [0,T]) \to C^{0,\A}(M\times [0,T])$ is a bounded map,
in view of the algebraic and regularity properties \eqref{DT} in Definition \ref{Schauder-def}.
Moreover, $\Phi$ quadratic in its argument, i.e. writing $\| \cdot \|_{k,\A}$ for the norm on $C^{k,\A}$
for any $k \in \N$, there exists a uniform $C>0$, such that by \eqref{quadratic}
(cf. \cite[Lemma 5.1]{ShortTime})
\begin{equation} \begin{split}
\forall w,w' \in C^{2,\A}: \quad  &\| \Phi(w) \|_{0,\A} \leq C \| w \|^2_{2,\A}, \\
&\| \Phi(w) -\Phi(w') \|_{0,\A} \leq C \left( \| w \|_{2,\A} + \| w' \|_{2,\A} \right) \| w-w' \|_{2,\A}.
\end{split} \end{equation}
Now a solution $v$ of \eqref{lin} (and hence also a solution $u=1+v$ of \eqref{eq:YamabeFlow}) is obtained
as a fixed point of the map 
\begin{equation}\label{FP}
C^{2,\A}(M\times [0,T]) \ni v \mapsto e^{t(n-1)\Delta_0} \left( - \frac{n-2}{4} S_0 + \Phi (v) \right)
\in C^{2,\A}(M\times [0,T]), 
\end{equation}
which is a contraction mapping on a subset of $C^{2,\A}(M\times [0,T])$ for $T>0$ sufficiently small\footnote{We need
to assume that $T>0$ is sufficiently small in order to control $e^{t(n-1)\Delta_0} (S_0)$.}, by 
\eqref{mapping-heat} in Definition \ref{Schauder-def}. One argues exactly as in 
\cite[Theorem 4.1]{ShortTime}. Note that the regularity of the scalar curvature $S$ along
the flow is then $S \in C^{0,\A}(M\times [0,T])$. \medskip

Note also, that the fixed point argument
is performed in a small ball around zero in $C^{2,\A}(M\times [0,T])$, and thus for $T>0$ sufficiently
small, the norm of $v$ is small. Hence $u=1+v$ is positive and bounded away from zero.
\medskip

The second statement improves regularity of $S$. By the regularity properties 
\eqref{LH} in Definition \ref{Schauder-def}, we conclude that $\rho, \partial_t \rho \in C^{0,\A}([0,T])$.
By \eqref{CT}, this implies that $\rho \in C^{2,\A}([0,T])$. We can now apply the mapping properties
\eqref{Q} and \eqref{R}\footnote{Here we use the assumption that $u$ is uniformly bounded away from zero, and 
that $1 \in C^{3,\A}$ by the algebraic properties of the Banach spaces.} 
in Definition \ref{Schauder-def} to obtain a solution $u' \in C^{3,\A}(M\times [0,T])$ to
\begin{equation}
\partial_t u' - (n-1) u^{-\frac{4}{n-2}} \Delta_0 u' = \frac{n-2}{4} \left( \rho u - S_0 u^{\frac{n-6}{n-2}} \right),
\quad u'(0) = 1. 
\end{equation}
The given solution $u \in C^{2,\A}$ satisfies the same equation, 
and we can prove by the weak maximum property \eqref{weak-max} of elements in $C^{2,\A}$, 
that $u\equiv u'$. Thus, indeed $u \in C^{3,\A}$ and hence $S \in C^{1,\A}$.
This is basically the argument also used in \cite[Proposition 4.8]{LongTime}.
\end{proof}

\begin{Rem}\label{SC2}
If we assume $Q : C^{2,\A} \to C^{4,\A}$ and $R : C^{4,\A}(M) \to C^{4,\A}$ in Definition 
\ref{Schauder-def}, as has been proved in \cite{LongTime}, then the condition $S_0 \in C^{2,\A}(M)$
implies by similar arguments as in Theorem \ref{short}, that any solution $u\in C^{2,\A}$
is actually in $C^{4,\A}$. This would lead to $S \in C^{2,\A}$, in particular the scalar curvature
would stay bounded along the flow. Here, we decided to require less in Definition 
\ref{Schauder-def}, assume less regularity for $S_0$, and conclude
boundedness of $S$ by Moser iteration methods instead.
\end{Rem}

\subsection{Regularity of the initial scalar curvature}

In view of Theorem \ref{short} we arrive at our final assumption on a
regularity of the initial scalar curvature $S_0$ with respect to the scale of 
Banach spaces in Definition \ref{Schauder-def}.

\begin{Assump}\label{S-bound}
Assuming that $(M,g_0)$ satisfies parabolic Schauder estimates, we also ask the
initial scalar curvature $S_0 \in C^{1,\A}(M)$. This implies in view of Theorem \ref{short}
$$
S \in C^0([0,T],H^1(M)) \cap C^1((0,T),H^1(M)).
$$
In particular, since the flow $u \in C^{2,\A}(M \times [0,T])$ is bounded from above and below
for $T>0$ sufficiently small, norms on the Sobolev space $H^1(M)$ with 
respect to $g_0$ and on the Sobolev space $H^1(M,g)$ with respect to 
$g=u^{\frac{4}{n-2}}g_0$, are equivalent. Thus $S$ lies in the Sobolev space $H^1(M,g)$
$$
S \in C^0([0,T],H^1(M,g)) \cap C^1((0,T),H^1(M,g)).
$$
\end{Assump} 

Our arguments below will use regularity of $S$ to show that given 
$S_0 \in L^q(M)$ for $q=\frac{n^2}{2(n-2)}=\frac{n}{2} +\frac{n}{n-2} > n/2$,
we may conclude by Moser iteration that $S \in L^\infty(M)$ for positive times. We close
this subsection with an observation, that on stratified spaces, 
$S_0 \in L^q(M)$ for $q > n/2$ and $S_0 \in L^\infty(M)$ basically carry the same geometric restriction. 
Indeed, consider a cone $(0,1) \times N$ over a Riemannian manifold $(N,g_0)$,
with metric $g_0 = dx^2 \oplus x^2 g_N + h$, where $h$ is smooth in $x\in [0,1]$ and 
$|h|_{\overline{g}} = O(x)$ as $x\to 0$, where we write $\overline{g}:= dx^2 \oplus x^2 g_N$. Then 

\begin{equation}
S_0 \sim \frac{\textup{scal}(g_N) - \dim N (\dim N -1)}{x^{2}} + O(x^{-1}), \quad \textup{as} \ x \to 0,
\end{equation}
where the higher order term $O(x^{-1})$ comes from the perturbation $h$. 
Both assumptions $S_0 \in L^\infty(M)$ and $S_0 \in L^q(M)$ for $q> n/2$
imply that the leading term of the metric $g_0$ is scalar-flat, i.e. $\textup{scal}(g_N) = \dim N (\dim N -1)$.

\subsection{The overarching strategy}
Studies of the Yamabe flow usually follow the following very rough pattern. 
One first argues that \eqref{eq:YamabeFlow} has a short-time solution. This is the step we have 
been concerned with in this section. This step doesn't invoke the sign 
of the Yamabe constant. \medskip

The next step is to show that the flow can be extended to all times. The way one does 
this is to assume the flow is defined for $t\in (0,T)$ for some maximal time $T<\infty$ and then 
derive a priori bounds on the solution $u$ and the scalar curvature $S$, showing 
that neither of them develop singularities as $t\to T$. One can thus keep flowing past 
$T$, establishing long-time existence. This is the step we are concerned with for the rest of the paper.

\section{The evolution of the scalar curvature and lower bounds}\label{lower-section}

In this section we derive a lower bound on the scalar curvature $S$ along the normalized 
Yamabe flow. We present an argument that does not require the maximum principle, but rather
the following assumptions

\begin{equation}\label{S-reg-ass}
\begin{split}
&S \in C^0([0,T],H^1(M,g)) \cap C^1((0,T),H^1(M,g)), \\
&C^\infty_c(M) \ \textup{is dense in } H^1(M), \\
&H^1(M) = H^1(M,g), \\
&Y(M,g_0) > 0.
\end{split}
\end{equation}

\noindent These properties follow from Assumptions \ref{Y-assump}, 
\ref{admissible-assump}, \ref{Schauder-assump} and \ref{S-bound}.

\begin{Lem}
Let $g= u^{\frac{4}{n-2}}g_0$ 
be a family of metrics evolving according to the normalized Yamabe flow 
\eqref{eq:YamabeFlow} satisfying\footnote{In fact here we do not require
$S\in C^0([0,T],H^1(M,g))$.} \eqref{S-reg-ass}. Then
$S$ evolves according to
\begin{equation}
\partial_t S-(n-1)\Delta S= S(S-\rho).
\label{eq:ScalarEvol}
\end{equation}
where $\Delta$ denotes the Laplacian with respect to the time-evolving metric $g$.
We write $S_{+}:= \max\{S,0\}$ and $S_{-}:= -\min\{S,0\}$. 
Then $S_{\pm}\in C^1((0,T), H^1(M,g))$ and satisfy
\begin{align}
&\partial_t S_{+}-(n-1)\Delta S_{+} \leq S_+(S_+-\rho),
\label{eq:S+Evol} \\
&\partial_t S_- -(n-1)\Delta S_- \leq -S_-(S_-+\rho).
\label{eq:S-Evol}
\end{align} 
\end{Lem}

\begin{Rem} The equation \eqref{eq:ScalarEvol}
is to be understood in the weak sense: 
for any compactly supported 
smooth test function $\phi \in C^\infty_c(M)$ we have
\begin{equation*}
\int_M \partial_t S \cdot \phi \, d\Vol_g + (n-1) \int_M ( \nabla S, \nabla \phi )_g 
d\Vol_g = \int_M S(S-\rho)  \cdot \phi \,  d\Vol_g.
\end{equation*}
Similarly for the partial differential inequalities \eqref{eq:S+Evol} and \eqref{eq:S-Evol}
and $\phi \geq 0$
\begin{equation*}
\int\limits_M \partial_t S_\pm  \cdot \phi \,  d\Vol_g + (n-1) \int\limits_M ( \nabla S_\pm, \nabla \phi )_g 
d\Vol_g \leq \pm \int\limits_M S_\pm(S_\pm \mp \rho)  \cdot \phi \,  d\Vol_g.
\end{equation*}
By \eqref{S-reg-ass}, $C^\infty_c(M)$
is dense in $H^1(M) = H^1(M,g)$.
Hence we can as well assume $\phi \in H^1(M,g)$ in the weak formulation above.
\end{Rem}

\begin{proof}
Equation \eqref{eq:ScalarEvol} is well-known, and can be deduced as follows. 
Write
$$
L_g := S-4 \, \frac{n-1}{n-2} \, \Delta,
$$ 
for the conformal Laplacian of the metric $g$. We write $L_0 \equiv L_{g_0}$.
When $g$ and $g_0$ are related by $g=u^{\frac{4}{n-2}} g_0$, then $L_g$ and $L_0$ are related by
\[
L_g(\cdot)=u^{-\frac{n+2}{n-2}} L_0(u \, \cdot ).
\]
In particular, $S=L_g(1)=u^{-\frac{n+2}{n-2}}L_0(u)$. Differentiate this equation weakly 
in time and use \eqref{eq:YamabeFlow} to replace $\partial_t u=-\frac{n-2}{4} (S-\rho)u$ to get
\[
\partial_t S=\frac{n+2}{4}(S-\rho)u^{-\frac{n+2}{n-2}}L_0(u)-\frac{n-2}{4}u^{-\frac{n+2}{n-2}}
L_0\left((S-\rho)u\right).
\]
Applying the transformation rule for $L$ we may rewrite this as
\begin{align*}
\partial_t S &=\frac{n+2}{4}(S-\rho)L_g(1)-\frac{n-2}{4}L_g(S-\rho)
\\ &=\frac{n+2}{4}(S-\rho)S+(n-1)\Delta S-\frac{n-2}{4}(S-\rho)S.
\end{align*}
This proves the formula \eqref{eq:ScalarEvol}. In order to derive 
the differential inequality for $S_+$, consider any $\varepsilon>0$ and define 
\[
\psi_{\varepsilon}(x):= \begin{cases} \sqrt{x^2+\varepsilon^2}-\varepsilon, & x\geq 0 
\\ 0, & x<0.\end{cases}
\]
For $v\in H^1(M,g)$ it is readily checked that $\psi_{\varepsilon}(v)\in H^1(M,g)$ and 
$\lim\limits_{\varepsilon\to 0} \psi_{\varepsilon}(v)=v_+$. Furthermore, we compute for the 
derivatives in case  $x>0$
$$
\psi_{\varepsilon}'(x)=\frac{x}{\sqrt{x^2+\varepsilon^2}},
\quad \psi''_{\varepsilon}(x)=\frac{\varepsilon^2}{(x^2+\varepsilon^2)^{\frac{3}{2}}}.
$$ 
These are both bounded for a fixed $\varepsilon > 0$, so the chain rule applies. 
Next up, we claim for any $v\in H^1(M,g)$ in the weak sense
\begin{equation}\label{psiv}
\Delta \psi_{\varepsilon}(v)\geq \frac{v}{\sqrt{v^2+\varepsilon^2}} \, \Delta v
\equiv \psi_{\varepsilon}'(v) \, \Delta v.
\end{equation}
This is seen as follows. Let $0\leq \xi\in C^{\infty}_c(M)$ be arbitrary and compute
\begin{align*}
\int_M \xi \,  \Delta \psi_{\varepsilon}(v)\, d\Vol_g&:=
 -\int_M ( \nabla \xi, \nabla \psi_{\varepsilon}(v) )_g \, d\Vol_g
 = -\int_M \frac{v}{\sqrt{v^2+\varepsilon^2}} ( \nabla \xi, \nabla v)_g \, d\Vol_g \\ 
 &= -\int_M \left( \nabla v, \nabla \left(\frac{v}{\sqrt{v^2+\varepsilon^2}} \, \xi\right)\right)_g\, d\Vol_g
 +\int_M \frac{\xi \varepsilon^2 \vert \nabla v\vert_g^2}{(v^2+\varepsilon^2)^{\frac{3}{2}}}\, d\Vol_g\\
&\geq -\int_M \left( \nabla v, \nabla \left(\frac{v}{\sqrt{v^2+\varepsilon^2}} \, \xi\right)\right)_g \, d\Vol_g   
\\ &=: \int_M \xi \frac{v}{\sqrt{v^2+\varepsilon^2}} \Delta v\, d\Vol_g.
\end{align*} 
This proves \eqref{psiv}, which allows us to deduce
\begin{align*}
\partial_t \psi_{\varepsilon}(S)-(n-1)\Delta \psi_{\varepsilon}(S)
\leq \, &\begin{cases} \psi_{\varepsilon}'(S) (\partial_t S-(n-1)\Delta S), & S\geq 0\\ 0, & S <0\end{cases}\\ \stackrel{\eqref{eq:ScalarEvol}}{=} \, &\begin{cases} \psi_{\varepsilon}'(S) S(S-\rho), & 
S\geq 0\\ 0, & S<0\end{cases} \\ = \ \,  & \frac{S}{\sqrt{S^2+\varepsilon^2}} S_+(S_+-\rho).
\end{align*}
Letting $\varepsilon\to 0$ results in \eqref{eq:S+Evol}.
To prove \eqref{eq:S-Evol}, observe that $S_- =S_+ -S$. Hence
\begin{align*}
\partial_t S_- -(n-1)\Delta S_- &=\partial_t S_+ -(n-1)\Delta S_+ -\left(\partial_t S -(n-1)\Delta S\right) \\ & 
\leq S_+(S-\rho)-S(S-\rho)=S_-(S-\rho),
\end{align*}
where we used \eqref{eq:ScalarEvol} and \eqref{eq:S+Evol} in the inequality step.
The only thing which remains to be observed is that $S_-  \cdot S=S_- (S_+ - S_-)=-S_- ^2$.  
\end{proof}

We can now derive lower bounds for $S$ by studying the evolution (in-) equalities above.
This is usually done by invoking the weak maximum principle for $S$, which  
is not available under the current assumptions \eqref{S-reg-ass}. Thus, we provide an alternative
novel argument, which does not use a maximum principle and which we could not find elsewhere in the 
literature.

\begin{Prop}
\label{Prop:LowerCurvBound2}
Let $g= u^{\frac{4}{n-2}}g_0$ 
be a family of metrics evolving according to the normalized Yamabe flow 
\eqref{eq:YamabeFlow} satisfying \eqref{S-reg-ass}. Then
\[
\norm{S_-}_{L^p(M,g)}( t) \leq e^{tn\frac{\rho(0)}{2p}} \norm{(S_0)_-} _{L^p(M)}
\]
holds for all $2\leq p\leq \infty$. In particular, if $(S_0)_- \in L^\infty(M)$, then 
$S_- \in L^{\infty}$ on $[0,T]$ with uniform bounds depending only on $T$ and $S_0$. 
Moreover, if $S_0\geq 0$, then $S\geq 0$ along the normalized Yamabe flow for all time.
\end{Prop}

\begin{proof}
The weak formulation of \eqref{eq:S-Evol} is that for any $0\leq \xi \in H^1(M,g)$ 
\begin{equation}
\int_M \xi \, \partial_t S_- \, d\Vol_g  + (n-1)\int_M ( \nabla S_-, \nabla \xi )_g\, d\Vol_g \leq 
-\int \xi \, S_-(S_-+\rho)\, d\Vol_g
\label{eq:WeakS-Evol}
\end{equation}
holds. A problem when manipulating this is of course that the chain rule fails 
to hold in general, so we use the same workaround as \cite[pp. 10-13]{ACM} 
(who in turn are following \cite[pp. 349-352]{Gursky}). Let $L>0$,  $\beta\geq 1$ and define 
\begin{equation}
\phi_{\beta,L}(x):= 
\begin{dcases} x^\beta, & x\leq L,\\ \beta L^{\beta-1}(x-L) + L^\beta, & x> L. \end{dcases}
\end{equation}
\begin{equation}\label{G-def}
G_{\beta,L}(x):= \int_0^x \phi_{\beta,L}'(y)^2\, dy=
\begin{dcases} \frac{\beta^2}{2\beta-1} x^{2\beta-1}, & x\leq L, 
\\ \beta^2 L^{2(\beta-1)} x -\frac{2\beta^2 L^{2\beta-1}(\beta-1)}{2\beta-1}, & x>L. 
\end{dcases} 
\end{equation}
Finally, we define $H_{\beta,L}(x):= \int\limits_0^x G_{\beta,L}(y)\, dy$ and conclude
\begin{equation*}
H_{\beta,L}(x) = \left\{ \begin{split}
&\frac{\beta x^{2\beta}}{2(2\beta-1)},  & x\leq L, \\
&\frac{\beta^2 L^{2(\beta-1)}}{2}(x^2-L^2)-\frac{2\beta^2 
L^{2\beta-1}(\beta-1)}{2\beta-1}(x-L)+\frac{\beta L^{2\beta}}{2(2\beta-1)},  & x>L. 
\end{split}\right.
\end{equation*}
The crucial features of these definitions are as follows
$$
\phi_{\beta,L}(x)\xrightarrow{L\to \infty} x^{\beta}, \quad 
G_{\beta,L}(x)\xrightarrow{L\to \infty} \frac{\beta^2}{2\beta-1} x^{2\beta-1},
\quad H_{\beta,L}(x)\xrightarrow{L\to \infty} \frac{\beta}{2(2\beta-1)} x^{2\beta}.
$$ 
These functions are also dominated by simpler expressions. For instance, $H_{\beta,L}(x)=\beta^2x^{2\beta}$ holds for all $L>0$ and $\beta\geq 1$, as one sees as follows. For $x\leq L$ there is nothing to show. For $x>L$, we first observe that 
\[H_{\beta,L}(x)=\frac{\beta^2}{2}L^{2(\beta-1)}x^2-\frac{2\beta^2(\beta-1)}{2\beta-1}L^{2\beta-1}x +\frac{\beta(\beta-1)}{2}L^{2\beta}.\]
Dropping the non-positive middle term and estimating by $x\geq L$ we find
\[H_{\beta,L}(x)\leq \frac{\beta^2}{2} x^{2\beta} +\frac{\beta(\beta-1)}{2}x^{2\beta} <\beta^2 x^{2\beta}.\]
Another important property is $\phi_{\beta,L}\in C^1(\R_+)$, with $\phi_{\beta,L}'\in L^{\infty}(\R_+)$ for all 
$L>0$, and so we may apply the chain rule to $\phi_{\beta,L}(S_-)$. Finally, since we are 
assuming a $C^1$ time-dependence, we have 
$\partial_t H_{\beta,L}(S_-)=(\partial_t S_-)G_{\beta,L}(S_-)$. 
We will use $\xi := G_{\beta,L}(S_-)$ as a test function in \eqref{eq:WeakS-Evol}. 
Note that by definition, $G_{\beta,L}(x)$ is linear for $x>L$ and hence $G_{\beta,L}(f)\in H^1(M,g)$ whenever $f\in H^1(M,g)$ (here we are also using that $\Vol(M)<\infty$). Then \eqref{eq:WeakS-Evol} implies
\begin{equation}\label{eq:NegEvol1}
\begin{split}
\int_M \partial_t H_{\beta,L}(S_-) \, d\Vol_g &\leq -(n-1) \int_ M \vert \nabla \phi_{\beta,L}(S_-)\vert_g^2 \, d\Vol_g 
\\ &-\int_M G_{\beta,L}(S_-) S_-(S_-+\rho)\, d\Vol_g.
\end{split}
\end{equation}
We then use \eqref{eq:DVol} to conclude
\begin{equation}\label{eq:NegEvol2}
\begin{split}
\int_M \partial_t H_{\beta,L}(S_-) \, d\Vol_g &=\partial_t \int_M H_{\beta,L}(S_-)\, d\Vol_g +\frac{n}{2} \int_M H_{\beta,L}(S_-)(S-\rho)\, d\Vol_g  \\ &=\partial_t \int_M H_{\beta,L}(S_-)\, d\Vol_g -\frac{n}{2} \int_M H_{\beta,L}(S_-)(S_- +\rho)\, d\Vol_g,
\end{split}
\end{equation}
where the last step uses $S H_{\beta,L}(S_-) \equiv (S_+-S_-)H_{\beta,L}(S_-)=-S_- H_{\beta,L}(S_-)$. 
Finally, we need a Sobolev inequality given to us by the positivity of the Yamabe constant, namely for any $f\in H^1(M,g)$ we have by the definition of $Y(M,g_0)$ (note that $Y(M,g_0) = Y(M,g)$ by conformal invariance)
\begin{equation}\label{Y-Sobolev}
Y(M,g_0) \norm{f}^2_{L^{\frac{2n}{n-2}}(M,g)} \leq 4\frac{n-1}{n-2} 
\norm{\nabla f}^2_{L^2(M,g)} +\int_M S \, f^2 \, d\Vol_g.
\end{equation}
We set $f=\phi_{\beta,L}(S_-)$. Observe
$\phi_{\beta,L}(S_-)^2 \, S = - \phi_{\beta,L}(S_-)^2 \, S_- $. Then \eqref{Y-Sobolev} implies
\begin{align}
&(n-1)\norm{\nabla \phi_{\beta,L}(S_-)}^2_{L^2(M,g)}\notag \\ &\geq \frac{n-2}{4} \,
Y(M,g_0)\norm{\phi_{\beta,L}(S_-)}^2_{L^{\frac{2n}{n-2}}(M,g)} 
+\frac{n-2}{4}\int_M \phi_{\beta,L}(S_-)^2 \, S_- \, d\Vol_g \notag \\ 
&\geq   \frac{n-2}{4}\int_M \phi_{\beta,L}(S_-)^2 \, S_- \, d\Vol_g.
\label{eq:NegEvol3}
\end{align}
Combining \eqref{eq:NegEvol1}, \eqref{eq:NegEvol2} and \eqref{eq:NegEvol3} yields
\begin{align*}
\partial_t \int_M H_{\beta,L}(S_-)\, d\Vol_g 
&\leq \int_M \left(\frac{n}{2} \, H_{\beta,L}(S_-)- G_{\beta,L}(S_-) S_- 
- \frac{n-2}{4} \, \phi_{\beta,L}(S_-)^2 \right)S_-\, d\Vol_g \\ 
 &+ \int_M  \rho\left(\frac{n}{2} \, H_{\beta,L}(S_-)- G_{\beta,L}(S_-) S_-\right)\, d\Vol_g.
\end{align*}
The claim is that the first group of terms on the right hand side is non-positive, 
which follows by a direct computation
\begin{equation*}
\begin{split}
&\frac{n}{2} \, H_{\beta,L}(x)-xG_{\beta,L}(x)-\frac{n-2}{4}\phi_{\beta,L}(x)^2 \\ 
&= \left\{ \begin{split} &\frac{(-1)}{4(2\beta-1)} \Bigl((4\beta+n)(\beta-1)+2\Bigr)x^{2\beta},  \quad x\leq L, 
\\ &\frac{L^{2\beta}}{4}\left( -2\beta^2\left(\frac{x}{L}\right)^2 -\frac{2(n-2)\beta(\beta-1)}{2\beta-1}
 \left(\frac{x}{L}\right) +(\beta-1)(n+2(\beta-1))\right),  x>  L.\end{split} \right.
\end{split}
\end{equation*}
In both cases one checks that the expressions are non-positive\footnote{For the $x\geq L$ case observe that the polynomial is negative for $x=L$, and the expression for $x>L$ clearly has a negative derivative. So the expression remains negative for $x>L$.}  for $\beta\geq 1$. Hence
using that $G_{\beta,L}(S_-) \geq 0$ and $\rho$ is non-increasing by \eqref{eq:rhoEvol}, we conclude
\[\partial_t \int_M H_{\beta,L}(S_-)\, d\Vol_g \leq \int_M  \frac{n\rho }{2}H_{\beta,L}(S_-) d\Vol_g\leq \frac{n\rho(0)}{2} \int_M H_{\beta,L}(S_-)\, d\Vol_g .\]
Integrating this shows
\[ \int_M H_{\beta,L}(S_-)\, d\Vol_g  (t) \leq e^{t\frac{n\rho(0)}{2}}   \int_M H_{\beta,L}(S_-)\, d\Vol_g  (t=0).\]
The conclusion will follow when we take the limit $L\to \infty$, which we can do for the following reason\footnote{This  
argument is applied several times, without writing out the details in the latter instances.}. 
On the left hand side we appeal to Fatou's lemma and the pointwise convergence of $H_{\beta,L}$;
\[\liminf_{L\to \infty} \int_M H_{\beta,L}(S_-)\, d\Vol_g \geq \int_M \liminf_{L\to \infty} H_{\beta,L}(S_-)\, d\Vol_g =\frac{\beta}{2(2\beta-1)}\int_M S^{2\beta}_- \, d\Vol_g.\]
The right hand side we deal with by the dominated convergence theorem. We showed above that $H_{\beta,L}(x)\leq \beta^2 x^{2\beta}$ holds for all $L>0$ and $\beta\geq 1$. Since we are assuming $(S_0)_- \in L^{\infty}(M)$, can use $\beta^2((S_0)_-)^{2\beta}$ as a dominating integrable function to deduce
\[\liminf_{L\to \infty} \int_M H_{\beta,L}((S_0)_-)\, d\mu =\lim_{L\to \infty} \int_M H_{\beta,L}((S_0)_-)\, d\mu=\frac{\beta}{2(2\beta-1)}\int_M ((S_0)_-)^{2\beta}\, d\mu\] 
Combined we conclude for $\beta \geq 1$
\[
\int_M S_-^{2\beta}\, d\Vol_g \leq  e^{t\frac{n\rho(0)}{2}}\int_M (S_0)_-^{2\beta} \, d\mu.
\]
This gives the conclusion when writing $2\beta=p$.
\end{proof}

\begin{Rem}
Let us again emphasize the novelty of this argument: it circumvents the maximum principle, 
and one only needs to know that $S\in C^1((0,T); H^1(M,g))\cap C^0([0,T];H^1(M,g))$, as 
assumed in \eqref{S-reg-ass}.
\end{Rem}

For completeness, let us also provide the classical widely known argument, cf.
\cite{Brendle}, using the weak maximum principle: we assume that $S$ satisfies \eqref{weak-max},
which is the case if $S \in C^{2,\A}(M\times [0,T])$. See Remark \ref{SC2} for 
conditions which ensure this regularity of $S$ along the flow.

\begin{Prop}
\label{Prop:LowerCurvatureBound}
Assume $S \in C^0(M\times [0,T])$ satisfies the weak maximum principle \eqref{weak-max}
and that $Y(M,g_0)>0$. Then $S$ admits a 
uniform lower bound
\[
S\geq  \min \, \{0,\inf_M S_0\}.
\]
\end{Prop}

\begin{proof}
By the weak maximum principle, we have for $S_{\min} := \inf\limits_M S$
\[\partial_t S_{\text{min}}\geq S_{\min}(S_{\min}-\rho).\]
If $S_{\min}$ is negative for all times, then the right hand side becomes positive, 
and we we get $S_{\min}\geq \inf\limits_{M} S_0$. 
If $S_{\min}$ is positive for all times, we can further estimate the right hand side using $\rho \leq \rho(0)$, cf. 
\eqref{eq:rhoEvol}. Dividing, we then get
\[
\frac{\partial_t S_{\min}}{S_{\min}(\rho(0)-S_{\min})} \geq -1.
\]
Integrating this differential inequality we find (writing $S_0^m := \inf_M S_0$)
\[
S_{\min}(t)\geq \frac{\rho(0) (S_0)_{\min} }{e^{\rho(0) t}
(\rho(0) - (S_0)_{\min}) + (S_0)_{\min}} \geq 0.
\]
If $S_{\min}$ changes the sign along the flow, the statement follows by
a combination of both estimates. 
\end{proof}

\section{Uniform bounds on the solution along the flow}\label{bounds-section}

The arguments of this section employ the following assumptions
\begin{equation}\label{u-bound-assump}
\begin{split}
&(M,g_0) \ \textup{is an admissible manifold}, \\
&u\in C^1((0,T);H^1(M))\cap C^0([0,T];H^1(M)), \\
&S\in C^1((0,T);H^1(M,g))\cap C^0([0,T];H^1(M,g)), \\
&H^1(M) = H^1(M,g), \ Y(M,g_0) > 0.
\end{split}
\end{equation}

\noindent These properties follow from Assumptions \ref{Y-assump}, 
\ref{admissible-assump}, \ref{Schauder-assump} and \ref{S-bound}.
\medskip

We begin with the upper bound on $u$, which follows easily from the 
 lower bound on the scalar curvature $S$, obtained in Proposition \ref{Prop:LowerCurvBound2}.
 
\begin{Prop}
\label{Prop:Upper-u-Bound}
Let $g= u^{\frac{4}{n-2}}g_0$ be a family of metrics, $u> 0$, such that 
\eqref{u-bound-assump} holds and the normalized Yamabe flow equation 
\eqref{eq:YamabeFlow} holds weakly. with $u(0)=1$. Assume furthermore that $(S_0)_-\in L^\infty(M)$, where $S_0$ is the scalar curvature of $g_0$. Then there exists some uniform constant $0<C(T)<\infty$, depending only 
on $T>0$ and $S_0$, such that $u\leq C(T)$ for all $t\in [0,T]$, $T<\infty$.
\end{Prop}

\begin{proof}
We have by \eqref{eq:YF} and \eqref{eq:rhoEvol} that
\[
\partial_t u =  -\frac{n-2}{4} (S-\rho) u \leq \frac{n-2}{4}(S_- +\rho)u \leq \frac{n-2}{4}(S_-+\rho(0))u.\]
By Proposition \ref{Prop:LowerCurvBound2} we have $\|S_- \|_{L^\infty(M)}\leq \norm{(S_0)_-}_{L^\infty(M)}$, 
and hence setting $C:= \frac{n-2}{4}\left( \norm{(S_0)_-}_{L^\infty(M)} +\rho(0)\right)$, we conclude
\[
\partial_t u \leq C u\implies u\leq e^{CT} u_0 =e^{CT}.
\]
\end{proof}

\noindent The lower bound is more intricate and in many ways more interesting. 
The argument will rely on the upper bound on $u$ and the lower bound on $S$. 
The proof will be a mixture and modification of the methods of \cite[pp. 20-21]{ACM} and \cite[pp. 221-222]{Brendle}.

\begin{Thm}
\label{Prop:Lower-u-Bound}
Let $g= u^{\frac{4}{n-2}}g_0$ be a family of metrics, $u > 0$, such that 
\eqref{u-bound-assump} holds and the normalized Yamabe flow equation 
\eqref{eq:YamabeFlow} holds weakly with $u(0)=1$. Assume furthermore that $(S_0)_-\in L^\infty(M)$
and that $S_0\in L^q(M)$ for some $q>\frac{n}{2}$. 
Then there exists some uniform constant $c(T)>0$, depending only 
on $T>0$ and $S_0$, such that $c(T)\leq u$ for all $t\in [0,T]$.
\end{Thm}

\begin{proof}
By combining \eqref{eq:YamabeFlow} and \eqref{eq:YF} we may solve away the term $\partial_t u$ and get
\[
-4 \, \frac{n-1}{n-2} \, \Delta_0 u =\left( u^{\frac{4}{n-2}} S-S_0\right)u.
\]
Using $(S_0)_-\in L^\infty(M)$ and $u \in L^\infty(M\times [0,T])$ by Proposition \ref{Prop:Upper-u-Bound}, 
we may define 
$$
P:= \frac{n-2}{4(n-1)} \left( S_0 +\| u\|^{\frac{4}{n-2}}_{L^{\infty}(M_T)} \norm{(S_0)_-}_{L^{\infty}(M)}\right)
\in L^q(M),
$$
Note that $P$ only depends on $S_0$ and $T$. 
Furthermore, Proposition \ref{Prop:LowerCurvBound2} yields
\begin{equation}
(-\Delta_0 + P) u \geq 0.
\label{eq:EllipticSuper}
\end{equation}
Let us explain the proof idea. 
Assume we can show that there is some $\delta>0$ such that $u^{-\delta}\in H^1(M)$ 
uniformly in $t \in [0,T]$. Then \eqref{eq:EllipticSuper} implies
\begin{equation}\label{minus-delta}
\begin{split}
(-\Delta_0 - \delta P ) u^{-\delta} &= \delta u^{-1-\delta} \Delta_0 u -\delta(1+\delta) u^{-2-\delta} \vert \nabla u\vert^2 -\delta Pu^{-\delta} \\ &=-\delta u^{-1-\delta}\left( -\Delta_0 + P\right) u -\delta(1+\delta) u^{-2-\delta} \vert \nabla u\vert^2\leq 0.
\end{split}
\end{equation}
This is precisely the setting of \cite[Proposition 1.8]{ACM}, 
which then concludes by Moser Iteration and Sobolev inequality \eqref{eq:Sobolev}
\[
\norm{u^{-\delta}}_{L^\infty(M)} \leq C \norm{u^{-\delta}}_{H^1(M)},
\]
where the constant $C>0$ depends on $\delta P$, hence only on $T$ 
and $S_0$, but not on $t$. Under our temporary  assumption \eqref{eq:EllipticSuper}, 
we thus get a uniform bound on $u^{-\delta}$, which gives a uniform lower bound on $u$.
\medskip

Hence we only need to show that $u^{-\delta}\in H^1(M)$ uniformly. Let $\varepsilon,\delta>0$ and (following \cite[pp. 20-21]{ACM}) define the functions $\psi_{\varepsilon}(u):= (u+\varepsilon)^{-\delta}$ and $\phi_{\varepsilon}(u):= (u+\varepsilon)^{-1-2\delta}$. These are both in $H^1(M)$ since $u$ is.  Using $\phi_{\varepsilon} $ as a test function in the weak formulation of \eqref{eq:EllipticSuper} we deduce
\[-\frac{1+2\delta}{\delta^2} \norm{ \nabla \psi_{\varepsilon}(u)}^2_{L^2(M)} +\int_M P u\phi_{\varepsilon}(u)\, d\mu\geq 0\]
and, using that $u\phi_{\varepsilon}(u)\leq \psi_{\varepsilon}(u)^2$ along with the Hölder inequality, we find
\begin{equation}
 \norm{ \nabla \psi_{\varepsilon}(u)}^2_{L^2(M)}\leq \frac{\delta^2}{1+2\delta} \norm{P}_{L^q(M)} \norm{\psi_\varepsilon(u)^2}_{L^{\frac{q}{q-1}}(M)}.
 \label{eq:H1BoundStep1}
 \end{equation}
Since $q>\frac{n}{2}$ we have $\frac{q}{q-1}<\frac{n}{n-2}$ and thus $\norm{\psi_\varepsilon(u)^2}_{L^{\frac{q}{q-1}}(M)}\leq \norm{\psi_{\varepsilon}(u)}^2_{L^{\frac{2n}{n-2}}(M)}$. By the Sobolev inequality \eqref{eq:Sobolev} we know 
\begin{equation}\label{psi1}
\norm{\psi_{\varepsilon}(u)}^2_{L^{\frac{2n}{n-2}}(M)} \leq A_0 \norm{\nabla \psi_{\varepsilon}(u)}_{L^2(M)}^2 +B_0 \norm{\psi_{\varepsilon}(u)}^2_{L^2(M)}.
\end{equation}
Next we need a Poincar\'{e} inequality. Let $B\subset M$ be a ball. Then, exactly as in 
\cite[Lemma 1.14]{ACM}, there exists a constant $C_B>0$ such that 
\begin{equation}\label{psi2}
\norm{f}^2_{L^2(M)} \leq C_B\left(\norm{\nabla f}^2_{L^2(M)} + \norm{f}^2_{L^2(B)}\right),
\end{equation}
holds for all $f\in H^1(M)$. Plugging \eqref{psi1} and \eqref{psi2} in \eqref{eq:H1BoundStep1} results in
\begin{equation*} 
\resizebox{0.95\hsize}{!}{
$\norm{ \nabla \psi_{\varepsilon}(u)}^2_{L^2(M)}\leq \frac{\delta^2}{1+2\delta}  \norm{P}_{L^q(M)} \left( (A_0 +B_0C_B) \norm{\nabla \psi_{\varepsilon}(u)}^2_{L^2(M)} +B_0 C_B\norm{\psi_{\varepsilon}(u)}^2_{L^2(B)}\right),$}
\end{equation*}
which is equivalent to the following inequality
\begin{equation*} 
\resizebox{0.95\hsize}{!}{
$\left(1- \frac{\delta^2}{1+2\delta}\norm{P}_{L^q(M)}(A_0 +B_0C_B) \right) \norm{ \nabla \psi_{\varepsilon}(u)}^2_{L^2(M)}\leq \frac{\delta^2}{1+2\delta} \norm{P}_{L^q(M)}B_0 C_B  \norm{\psi_{\varepsilon}(u)}^2_{L^2(B)}.$}
\end{equation*}
Choosing $\delta>0$ small enough so that the left hand side becomes positive, we get a uniform (meaning now both $t$- and $\varepsilon$-independent) bound on $ \norm{ \nabla \psi_{\varepsilon}(u)}_{L^2(M)}$ if we can get  a uniform bound on  $\norm{\psi_{\varepsilon}(u)}_{L^2(B)}$. The uniform bound on $\norm{\psi_{\varepsilon}(u)}_{L^2(B)}$ will come 
from the local theory for elliptic supersolutions. Observe that since $u$ satisfies $\eqref{eq:EllipticSuper}$, 
$u^{\frac{2n}{n-2}}$ satisfies (by the same computation as in \eqref{minus-delta})
\[
-\Delta_0 u^{\frac{2n}{n-2}} +\frac{2n}{n-2} P u^{\frac{2n}{n-2}}\geq 0.
\]
 Let $R>0$ be such that $B_{4R}(x)\subset M$ for some $x\in M$. Then, according to \cite[Theorem 8.18, p. 194]{GT} the following weak Harnack inequality holds on  $B_{2R}(x)$, namely there is a constant $C>0$ independent of $u$ but depending on $g_0$, $R$ and $n$ such that 
 \begin{equation}
\Vol_g(B_{2R}(x)) \equiv \norm{u^{\frac{2n}{n-2}}}_{L^1(B_{2R}(x))}\leq C \inf_{B_R(x)} u^{\frac{2n}{n-2}},
 \label{eq:LocalHarnack}
 \end{equation}
where in the first identification we recalled $d\Vol_g=u^{\frac{2n}{n-2}} d\mu$. 
By admissibility of $(M,g_0)$, the assumption \eqref{eq:Exhaustion} holds and we may take a collection of balls $B_{4R_i}(x_i)\subset M$, indexed by $i=1,\cdots, N< \infty$, with the property that 
\begin{equation}\label{balls}
\left(1-\Vol_{g_0}\left(\bigcup_{i=1}^N  B_{2R_i}(x_i)\right)\right)\norm{u}^{\frac{2n}{n-2}}_{L^{\infty}(M_T)}<1.
\end{equation}
Let $C_i$ be the constant in \eqref{eq:LocalHarnack} for the ball $B_{2R_i}(x_i)$. 
By summing all the individual inequalities \eqref{eq:LocalHarnack} for each $i=1,\cdots, N$, we have
 \[\sum_{i=1}^N \Vol_g(B_{2R_i}(x_i)) \leq \sum_{i=1}^N C_i \inf_{B_{R_i}(x_i)} u^{\frac{2n}{n-2}} \leq NC \max_i \left( \inf_{B_{R_i}(x_i)}  u^{\frac{2n}{n-2}} \right)\]
 with $C:= \max_i C_i$. The left hand side we can bounded from below by
 \begin{align*}
 \sum_{i=1}^N \Vol_g(B_{2R_i}(x_i)) &\geq \Vol_g\left( \bigcup_{i=1}^N B_{2R_i}(x_i)\right) \\ &=1-\Vol_g\left(M\setminus  \bigcup_{i=1}^N B_{2R_i}(x_i)\right)\\ 
 &\geq 1-\Vol_{g_0}\left(M\setminus \bigcup_{i=1}^N  B_{2R_i}(x_i)\right)\norm{u}^{\frac{2n}{n-2}}_{L^{\infty}(M_T)}=:c
 \end{align*}
which is positive by choice of the balls subject to \eqref{balls}. Thus
 \[
 0 < c \leq NC \max_i \left( \inf_{B_{R_i}(x_i)}  u^{\frac{2n}{n-2}} \right).
 \]
 This shows that there has to be a ball $B_{R_i}(x_i)$ with $u$ uniformly bounded from below by $c(T)>0$ for $t\in [0,T]$. On this ball we thus get a uniform bound $\psi_{\varepsilon}(u) \geq c(T)^{-\delta}$, which gives our desired $t$- and $\varepsilon$-independent bound on $\norm{\psi_{\varepsilon}(u)}_{L^2(B)}^2$, and thereby that $u^{-\delta}\in H^1(M)$ uniformly.
\end{proof}

\begin{Cor}
\label{Cor:Sob}
Under the conditions of Theorem \ref{Prop:Lower-u-Bound}, 
one can find uniform constants $0<A(T),B(T)<\infty$, depending only on $T>0$ and initial scalar curvature $S_0$ (but not 
dependent on $t$), such that for all $f\in H^1(M,g)$
\begin{equation}
\norm{f}^2_{L^{\frac{2n}{n-2}}(M,g)}\leq A(T) \norm{\nabla f}^2_{L^2(M,g)}+B(T)\norm{f}^2_{L^2(M,g)},
\label{eq:gSobolev}
\end{equation}
i.e. \eqref{eq:Sobolev} holds for the time-dependent metric but with time-independent constants. 
\end{Cor}
\begin{proof}
Due to \eqref{eq:Sobolev} we have for all $f\in H^1(M) = H^1(M,g)$
\[\norm{f}^2_{L^\frac{2n}{n-2}(M,g_0)} \leq A_0 \norm{\nabla f}^2_{L^2(M,g_0)} +B_0 \norm{f}^2_{L^2(M,g_0)}.\]
Using $g= u^{\frac{4}{n-2}}g_0$ we conclude a similar estimate with respect to $g$
\begin{equation} \begin{split}
&\norm{f}^2_{L^{\frac{2n}{n-2}}(M,g)}\leq A(T) \norm{\nabla f}^2_{L^2(M,g)} + B(T)\norm{f}^2_{L^2(M,g)}, \\
&\textup{where} \ A(T):= A_0 \frac{(\sup_{M_T} u)^2}{(\inf_{M_T} u)^{2}}, \quad 
B(T):= B_0 \frac{(\sup_{M_T} u)^2}{(\inf_{M_T} u)^{\frac{2n}{n-2}}}.
\end{split} \end{equation}
Now the statement follows, since $u, u^{-1} \in L^\infty(M\times [0,T])$
by Proposition \ref{Prop:Upper-u-Bound} and Theorem \ref{Prop:Lower-u-Bound}.
\end{proof}

\noindent We shall need this Sobolev inequality \eqref{eq:gSobolev} when we tackle the 
upper bound on the scalar curvature $S$ in section \ref{upper-section}.

\section{Upper bound on the scalar curvature along the flow}\label{upper-section}

The arguments of this section employ the following assumptions
\begin{equation}\label{Smax-bound-assump}
\begin{split}
&(M,g_0) \ \textup{is an admissible manifold}, \\
&S\in C^1((0,T);H^1(M,g))\cap C^0([0,T];H^1(M,g)), \\
&\textup{The Sobolev inequality } \eqref{eq:gSobolev} \textup{ holds},\\
&H^1(M) = H^1(M,g), \ Y(M,g_0) > 0, \\
&C^\infty_c(M) \ \textup{is dense in } H^1(M).
\end{split}
\end{equation}

\noindent These properties follow from Assumptions \ref{Y-assump}, 
\ref{admissible-assump}, \ref{Schauder-assump}, \ref{S-bound}, as in the 
previous section. The Sobolev inequality  \eqref{eq:gSobolev} holds under the same
assumptions in view of Corollary \ref{Cor:Sob}. In this section we use 
\eqref{Smax-bound-assump} to show a uniform upper bound on the scalar curvature. 
More precisely, we will show the following result.

\begin{Thm}
\label{Prop:CurvatureBound}
Let $S$ evolve according to 
\eqref{eq:ScalarEvol} with initial curvature $S_0\in L^{\frac{n^2}{2(n-2)}}(M)$, 
and its negative part $(S_0)_- \in L^{\infty}(M)$. Then, assuming \eqref{Smax-bound-assump} holds, 
there exists a uniform constant $0<C(T)<\infty$, depending only on $T>0$ and $S_0$, such that 
\[\norm{S}_{L^\infty\left(M\times \left[\frac{T}{2},T\right]\right)}\leq C(T).\]
\end{Thm}

The proof proceeds in two steps. First step is to prove an $L^{\frac{n^2}{2(n-2)}}(M,g)$-norm
bound on $S$, uniform in $t\in [0,T]$. That uniform bound rests on a chain of arguments of 
\cite[Lemma 2.2, Lemma 2.3, Lemma 2.5]{Brendle} (also to be found in \cite[Lemma 3.3]{SS}) 
that apply in our setting as well. In the second step we perform a Moser iteration argument by 
following \cite{MCZ12}. Our proofs are close to those in \cite{Brendle} with some additional arguments 
due to lower regularity. 

\begin{Lem}\label{Lem1} Under the conditions of Theorem \ref{Prop:CurvatureBound},
there exists for any finite $T>0$ a uniform constant $0<C(T)<\infty$, depending 
only on $T$ and $S_0$, such that for all $t\in [0,T]$ we have the following estimate
\footnote{\, Below, we will denote all uniform positive constants, depending only on $T$ 
and $S_0$ either by $C(T)$ or $C_T$, unless stated otherwise.}
\begin{equation}
\int_0^T \left( \int_M S^{\frac{n^2}{2(n-2)}} \, d\Vol_g \right)^{\frac{n-2}{n}}\, dt \leq C(T),
\quad \norm{S}_{L^{\frac{n}{2}}(M,g)} \leq C,
\label{eq:Claim}
\end{equation}
where the second constant $C$ only depends on $S_0$, not on $T$.
\end{Lem}

\begin{proof}
It suffices to prove the statement for $S_+$ and $S_-$ individually.
By Proposition \ref{Prop:LowerCurvBound2}, the statement holds for the negative part $S_-$. 
Thus we only need to prove the claim for $S_+$. We may therefore assume without loss of generality that $S\geq 0$, so that 
$S \equiv S_+$, and use \eqref{eq:S+Evol} as the evolution equation. 
\medskip

The claim will follow from the evolution equation \eqref{eq:ScalarEvol}, but we have to argue a bit differently depending on whether $3\leq n\leq 4$ or $n>4$. The idea is the same in all dimensions $n\geq 3$ however. Let us start with $3\leq n\leq 4$.
Fix any $\sigma>0$, and set $\beta=\frac{n}{4}$. Since $\beta\leq 1$, the function $x\mapsto (x+\sigma)^{\beta}$ is in $C^1[0,\infty)$ with bounded derivative. Thus, we may apply the chain rule to $(S+\sigma)^{\beta}$ and conclude $(S+\sigma)^\beta\in C^1([0,T];H^1(M,g))$. We use 
$\frac{\beta^2}{2\beta-1}(S+\sigma)^{2\beta-1}$ as a test function with $\beta=\frac{n}{4}$ in the weak formulation of \eqref{eq:S+Evol}, which yields the following inequality
\begin{align*}
&\frac{\beta^2}{2\beta-1}\int_M (S+\sigma)^{2\beta-1} \partial_t(S+\sigma)\, d\Vol_g +(n-1)\int_M \vert \nabla (S+\sigma)^{\beta}\vert^2\, d\Vol_g \\
& \leq \frac{\beta^2}{2\beta-1} \int_M S(S-\rho)(S+\sigma)^{2\beta-1}\, d\Vol_g.
\end{align*}
Using \eqref{eq:DVol} results in
\begin{align*}
&\frac{\beta}{2(2\beta-1)} \partial_t \int_M (S+\sigma)^{2\beta}\, d\Vol_g +(n-1)\int_M \vert \nabla (S+\sigma)^{\beta}\vert^2\, d\Vol_g \\
&\leq \frac{\beta}{2\beta-1}\int_M \beta  S(S-\rho)(S+\sigma)^{2\beta-1} -\frac{n}{4}(S-\rho)(S+\sigma)^{2\beta}\, d\Vol_g\\
&= -\frac{\beta^2 \sigma}{2\beta-1} \int_M (S-\rho)(S+\sigma)^{2\beta-1}\, d\Vol_g\\
&=-\frac{\beta^2\sigma}{2\beta-1} \int_M (S+\sigma-\rho)(S+\sigma)^{2\beta-1}\, d\Vol_g +\frac{\sigma^2 \beta^2}{2\beta-1} \int_M (S+\sigma)^{2\beta-1}\, d\Vol_g\\
&\leq \frac{\sigma \beta^2(\sigma+\rho(0))}{2\beta-1} \int_M (S+\sigma)^{2\beta-1}\, d\Vol_g \leq  \frac{\sigma \beta^2(\sigma+\rho(0))}{2\beta-1} \int_M (S+\sigma)^{2\beta}\, d\Vol_g,
\end{align*}
where the first equality is due to $\beta=\frac{n}{4}$, the penultimate inequality uses $\rho(0)\geq \rho(t)$,
and the final inequality is due to Hölder with $p=\frac{\beta}{\beta-1}$ and $q=\beta$.
We want to integrate this inequality in time. Note that any inequality of the 
form $\partial_t w(t) +a(t) \leq b w(t)$ with $a(t)\geq 0$ yields  $\partial_t w \leq b w$ and hence
$w(t) \leq e^{bt} w(0)$. Plugging this estimate into the original differential inequality leads to 
$\partial_t w +a \leq be^{bt} w(0)$. Integrating the latter inequality in time yields $w(t) + \int_0^t a(s) ds \leq e^{bt} w(0)$.
We therefore conclude that
\begin{equation} 
\resizebox{0.99\hsize}{!}{
$\int\limits_M (S+\sigma)^{\frac{n}{2}} \, d\Vol_g (T) +\frac{4(n-2)(n-1)}{n} \int\limits_0^T 
\int\limits_M \vert\nabla (S+\sigma)^{\frac{n}{4}}\vert^2\, d\Vol_g \leq 
e^{\frac{\sigma(\sigma+\rho(0)) nT}{2}} \int\limits_M (S_0+\sigma)^{\frac{n}{2}}\, d\mu.$}
\label{eq:n/2Bound}
\end{equation}
This is for any $\sigma>0$. Sending $\sigma\to 0$, using Fatou's lemma on the left hand side and the monotone convergence theorem on the right hand side yields (on dropping the non-negative term with $\nabla S$)
\[\int_M S^{\frac{n}{2}} \, d\Vol_g(T)\leq \int_M S_0^{\frac{n}{2}} \, d\mu. \] 
This yields our uniform $L^{\frac{n}{2}}(M,g)$-bound on $S$ in \eqref{eq:Claim}. 
Returning to \eqref{eq:n/2Bound}, we appeal to the Sobolev inequality \eqref{eq:gSobolev} to deduce
\[\int_0^T \norm{(S+\sigma)^{\frac{n}{4}}}^2_{L^{\frac{2n}{n-2}}(M,g)} \, dt\leq \left( \frac{A(T)n}{4(n-2)(n-1)}+TB(T)\right)e^{\frac{\sigma(\sigma+\rho(0)) nT}{2}} \int_M (S_0+\sigma)^{\frac{n}{2}}\, d\mu,\]
hence also
\[\int_0^T \left(\int_M \vert S\vert ^{\frac{n^2}{2(n-2)}}\, d\Vol_g\right)^{\frac{n-2}{n}} \, dt\leq C(T).\]
This proves the claim for $3\leq n\leq 4$. \medskip

\noindent For $n> 4$ the claim will follow similarly, but the above test function does not have bounded derivative for $n>4$, and we neither know that it is in $H^1$, nor do we know that the chain rule applies.  We therefore argue similarly to the the proof of Proposition \ref{Prop:LowerCurvBound2}, where we introduced the functions $\phi_{\beta,L}$, $G_{\beta,L}$ and
$H_{\beta,L}$. We set here again $\beta = n/4$. Using $G_{\beta,L}(S)$ as a test function in \eqref{eq:ScalarEvol} we find
\begin{align*}
\int\limits_M G_{\beta,L}(S) (\partial_t S) \, d\Vol_g+(n-1) \int\limits_M \left \vert \nabla \phi_{\beta,L}(S)\right\vert^2\, d\Vol_g 
\leq \int\limits_M S(S-\rho) G_{\beta,L}(S)\, d\Vol_g. 
\end{align*}
Using the evolution equation \eqref{eq:DVol} for the volume form, we have
\begin{align}
&\partial_t \int_M H_{\beta,L}(S)\, d\Vol_g +(n-1)\int_M \vert \nabla \phi_{\beta,L}(S)\vert^2\, d\Vol_g \notag \\ 
&\leq \int_M (S-\rho)\left(SG_{\beta,L}(S)-\frac{n}{2} H_{\beta,L}(S)\right)\, d\Vol_g.
\label{eq:CurvatureEvolStep}
\end{align}
One readily checks from the definition of $G_{\beta,L}$ and $H_{\beta,L}$ in Proposition \ref{Prop:LowerCurvBound2} that
\begin{align}
&x \, G_{\beta,L}(x)-\frac{n}{2} \, H_{\beta,L}(x) \notag\\
&=\begin{dcases}
\frac{\beta}{2\beta-1} x^{2\beta} \left(\beta-\frac{n}{4}\right),& x\leq L, \\
 \beta^2 L^{2\beta} \left( \left(1-\frac{n}{4}\right)\left(\frac{x}{L}\right)^2 +
 \frac{2(\beta-1)}{2\beta-1} \left(\frac{n}{2}-1\right)\frac{x}{L} -\frac{n(\beta-1)}{4\beta}\right),& x>L, 
\end{dcases}
\label{eq:HGComp}
\end{align}
and from this one sees that $xG_{\beta,L}(x)-\frac{n}{2}H_{\beta,L}(x)\leq 0$ for $\beta = \frac{n}{4}$ and $n\geq  4$ as follows. For $x\leq L$ there is nothing to show. For $x>L$, notice that 
\begin{align*}
&\beta^2 L^{2\beta} \left( \left(1-\frac{n}{4}\right)\left(\frac{x}{L}\right)^2 +\frac{2(\beta-1)}{2\beta-1} \left(\frac{n}{2}-1\right)\frac{x}{L} -\frac{n(\beta-1)}{4\beta}\right)\\
&=-\beta^2(\beta-1) L^{2\beta} \left(\frac{x}{L}-1\right)^2\leq 0
 \end{align*}
 where we have substituted $n=4\beta$ and recognized a square.\footnote{This is the point where we need $n\neq 3$, since in this case $\beta-1<0$ and the above expression fails to be negative for $x>L$.} 
 Hence, using again that $\rho$ is non-increasing along the flow, we conclude that
\begin{equation}
\partial_t \int_M H_{\beta,L}(S)\, d\Vol_g +(n-1)\int_M \vert \nabla \phi_{\beta,L}(S)\vert^2\, d\Vol_g \leq 0
\notag
\end{equation}
holds for any $L\geq \rho(0)$.
This is a differential inequality of the same kind as in the above  $3\leq n\leq 4$ case. Integrating it we deduce for any $t\in [0,T]$\begin{equation}\label{auxiliary1}
\begin{split}
\int\limits_M H_{\beta,L}(S)\, d\Vol_g (T) &+
(n-1) \int\limits_0^T \int\limits_M \vert \nabla \phi_{\beta,L}(S)\vert^2\, d\Vol_g \, dt
\\ &\leq \int\limits_M H_{\beta,L} (S_0) \, d\mu,
\end{split}
\end{equation}
Using $\beta=\frac{n}{4}$ and letting $L\to \infty$ this
yields, using Fatou's lemma and dominated convergence exactly as in the final step of the 
proof of Proposition \ref{Prop:LowerCurvBound2} (neglecting the positive second summand on the left hand side of 
\eqref{auxiliary1}) the following inequality
\begin{equation}\label{auxiliary2}
\| S \|_{L^{\frac{n}{2}}(M,g)} = \left( \int\limits_M S^{\frac{n}{2}}\, d\Vol_g \right)^{\frac{2}{n}}
\leq \left(   \int\limits_M \vert S_0\vert^{\frac{n}{2}} \, d\mu \right)^{\frac{2}{n}} \equiv C,
\end{equation}
where the constant $C(T)>0$ depends only
on $T$ and $S_0$. This yields the second estimate in \eqref{eq:Claim}, in the case of $n> 4$.
For the first estimate in \eqref{eq:Claim} note that $\phi_{\beta,L}(S) \in H^1(M,g)$\footnote{
Note that a priori we do not know if $S^{n/4} \in H^1(M;g)$ and thus cannot 
directly apply the Sobolev inequality \eqref{eq:gSobolev} to $S^{n/4}$. However we do know that
$\phi_{\beta,L}(S) \in H^1(M,g)$, since $\phi_{\beta,L}(x)$ is linear for $x>L$ and 
$S \in H^1(M,g)$ for each fixed time argument.} and thus by 
\eqref{eq:gSobolev} and \eqref{auxiliary1} we deduce
\begin{align*}
&\int_0^T \left( \int_M \vert \phi_{\beta,L}(S) \vert^{\frac{2n}{n-2}} \, d\Vol_g \right)^{\frac{n-2}{n}}\, dt
\\ &\leq A(T) \int_0^T \int_M \vert \nabla \phi_{\beta,L}(S)\vert^2 d\Vol_g \, dt 
+ B(T)  \int_0^T \int_M \vert \phi_{\beta,L}(S)\vert^2 d\Vol_g \, dt 
\\ &\leq \frac{A(T)}{n-1} \left( \int\limits_M H_{\beta,L} (S_0) \, d\mu
- \int\limits_M H_{\beta,L}(S)\, d\Vol_g (T) \right) 
\\ &+ B(T)  \int_0^T \int_M \vert \phi_{\beta,L}(S)\vert^2 d\Vol_g \, dt.
\end{align*}
Thus, letting $L\to \infty$ we conclude using Fatou's lemma and dominated convergence as 
before, with $\beta = n/4$ and \eqref{auxiliary2}
\begin{equation}\label{eq:IntegratedBound}
\begin{split}
&\int_0^T \left( \int_M S^{\frac{n^2}{2(n-2)}} \, d\Vol_g \right)^{\frac{n-2}{n}}\, dt
\\ &\leq \left(\frac{nA(T)}{4(n-1)(n-2)}  
+ B(T) T\right) \int_M  \vert S_0\vert^{\frac{n}{2}} d\mu  \equiv C(T),
\end{split}
\end{equation}
where the uniform constant $C(T)>0$ depending only on $T$ and $S_0$. This proves the
first estimate in \eqref{eq:Claim} for $n> 4$.
\end{proof}

\begin{Lem} Under the conditions of Theorem \ref{Prop:CurvatureBound},
there exists for any finite $T>0$ a uniform constant $0<C(T)<\infty$, depending 
only on $T$ and $S_0$, such that for all $t\in [0,T]$ we have the following estimate
\[
\int_M \left\vert S \right\vert^{\frac{n^2}{2(n-2)}}\, d\Vol_g \leq C(T).
\]
\label{Lem:UnifCurvBound}
\end{Lem}

\begin{proof}
As in the previous lemma we have to split the argument into cases based on the dimension. We first show the statement for $n\geq 4$.
We will again use the inequality \eqref{eq:CurvatureEvolStep}. However 
while in Lemma \ref{Lem1} we have set $\beta=n/4$, here we will use the inequality \eqref{eq:CurvatureEvolStep} for
$\beta=\frac{n^2}{4(n-2)}$. For this choice of $\beta$ the expression $xG_{\beta,L}(x)-\frac{n}{2}H_{\beta,L}(x)$ 
is not necessarily non-positive any longer and we estimate it against a new approximation function
\begin{equation} 
f_{\beta,L}(x) := \left\{
\begin{split}
& \beta x^{2\beta}, & x\leq L, \\ 
& n\beta^2 L^{2\beta-1}x,  & x>L.
 \end{split} \right.
 \end{equation}
By \eqref{eq:HGComp} one sees that $xG_{\beta,L}(x)-\frac{n}{2}H_{\beta,L}(x)\leq f_{\beta,L}(x)$ 
holds for all $\beta\geq 1$ and $L>0$, in case $n\geq 4$. One important aspect to notice is that 
$f_{\beta,L}(x)$ is linear in $x$ for $x>L$, as opposed to quadratic for $H_{\beta,L}(x)$ and $xG_{\beta,L}(x)$.  
This will become important below. Returning to \eqref{eq:CurvatureEvolStep}, and applying \eqref{eq:gSobolev} 
to the term $\norm{\nabla \phi_{\beta,L}(S)}^2_{L^2(M,g)}$, we find after some reshuffling
\begin{equation}\label{H-estimate2}
\begin{split}
&\partial_t \norm{H_{\beta,L}(S)}_{L^1(M,g)} \\ 
&\leq (n-1)\frac{B(T)}{A(T)} \norm{\phi_{\beta,L}(S)^2}_{L^1(M,g)}
-\frac{(n-1)}{A(T)} \norm{\phi_{\beta,L}(S)}_{L^{\frac{2n}{n-2}}(M,g)}^2 \\
&+\rho(0)\norm{SG_{\beta,L}(S)-\frac{n}{2}H_{\beta,L}(S)}_{L^1(M,g)}
+\norm{Sf_{\beta,L}(S)}_{L^1(M,g)}.
\end{split}\end{equation} 
A straightforward computation shows for all $\beta\geq 1$ and $L>0$
\begin{equation}
12\beta H_{\beta,L}(x)\geq \phi_{\beta,L}(x)^2, \quad 4\beta H_{\beta,L}(x)\geq xG_{\beta,L},
\label{eq:H-Comparison}
\end{equation}
holds, and here is a way of seeing this. For $x\leq L$ these are both obvious from the definitions, so we look at $x>L$. One first notices that 
\begin{align*}
\phi_{\beta,L}(x)^2 &=\beta^2L^{2(\beta-1)}x^2-2\beta(\beta-1)L^{2\beta-1}x +(\beta-1)^2L^{2\beta}\\
&\leq \beta^2 L^{2(\beta-1)} (x^2+2L^2)
\leq 3\beta^2 L^{2(\beta-1)} x^2, 
\end{align*}
where the first inequality comes from dropping the non-positive linear term and estimating $1\leq \beta$, 
and the final inequality is simply $L^2<x^2$.
We similarly estimate $H_{\beta,L}(x)$ from below for $x>L$ and find
\begin{equation}\label{eq:H-Manipulations}
\begin{split}
H_{\beta,L}(x)&=\left(\frac{\beta^2}{2}\left(\frac{x}{L}\right)^2-\frac{2\beta^2(\beta-1)}{2\beta-1}\left(\frac{x}{L}\right) +\frac{\beta(\beta-1)}{2}\right) L^{2\beta} \\
&\geq \left(\frac{\beta^2}{2(2\beta-1)} \left(\frac{x}{L}\right)^2 +\frac{\beta(\beta-1)}{2}\right)L^{2\beta}
\geq \frac{\beta^2}{2(2\beta-1)} x^2L^{2(\beta-1)},
\end{split}\end{equation}
where the first inequality uses $-\frac{x}{L} \geq -\frac{x^2}{L^2}$ and the second inequality comes from dropping the non-negative constant term. Using these two estimates one readily sees that 
\[12\beta H_{\beta,L}(x)\geq \frac{2\beta}{2\beta-1} 3\beta^2 x^2 L^{2(\beta-1)} \geq 3\beta^2 x^2 L^{2(\beta-1)} \geq \phi_{\beta,L}(x)^2,\]
showing half of the claim in \eqref{eq:H-Comparison}. To see the other half, first observe that (for $x>L$), we have $xG_{\beta,L}(x)\leq \beta^2 x^2 L^{2(\beta-1)}$ by dropping the non-positive term in \eqref{G-def}. Using \eqref{eq:H-Manipulations} again we deduce
\[4\beta H_{\beta,L}(x)\geq \frac{2\beta}{2\beta-1} \beta^2 x^2 L^{2(\beta-1)}\geq \beta^2x^2L^{2(\beta-1)} \geq xG_{\beta,L}(x).\]
This finishes the proof of \eqref{eq:H-Comparison}, so we arrive by overestimating the right hand side of \eqref{H-estimate2} at
the following inequality
\begin{equation}\label{eq:HEvolStep}
\begin{split}
\partial_t \norm{H_{\beta,L}(S)}_{L^1(M,g)}  &\leq C_T \norm{H_{\beta,L}(S)}_{L^1(M,g)} 
\\ &+ \norm{Sf_{\beta,L}(S)}_{L^1(M,g)} -\frac{(n-1)}{A(T)} \norm{\phi_{\beta,L}(S)}_{L^{\frac{2n}{n-2}}(M,g)}^2,
\end{split}\end{equation} 
where the uniform constant $C_T>0$ is explicitly given by 
$$
C_T := 12(n-1)\beta \, \frac{A(T)}{B(T)} + \rho(0)\left(\frac{n}{2}+4\beta\right).
$$ 
Introduce the non-negative, real function $F_{\beta,L}$ via 
$$
F_{\beta,L}(x) := \left( xf_{\beta,L}(x) \right)^{\frac{1}{2\beta+1}}.
$$
Assume $\beta>\frac{n}{4}$, which holds e.g. for $\beta=\frac{n^2}{4(n-2)}$. 
Set $\alpha:= \frac{n}{4\beta}<1$ and choose any $\delta>0$. Observe that by the 
H\"older inequality in the first estimate and the Young inequality in the second, we obtain
\begin{equation}\label{eq:HEvolStep2}
\begin{split}
&\norm{F_{\beta,L}(S)^{2\beta+1}}_{L^1(M,g)} \\
 &\leq \norm{\frac{}{}F_{\beta,L}(S)}^{2\alpha \beta}_{L^{\frac{2n\beta}{n-2}}(M,g)} 
 \norm{ \frac{}{}F_{\beta,L}(S)}^{1+2(1-\alpha)\beta}_{L^{2\beta}(M,g)}  \\
 &\leq \delta \alpha \norm{\frac{}{}F_{\beta,L}(S)^\beta}^2_{L^{\frac{2n}{n-2}}(M,g)} 
 +\delta^{-\frac{\alpha}{1-\alpha}}(1-\alpha) \norm{\frac{}{}F_{\beta,L}(S)}^{\frac{1}{1-\alpha}+2\beta}_{L^{2\beta}(M,g)}.
\end{split}\end{equation} 
 These norms are finite for finite $L>0$, as one can see as follows: the claim is clear for $S \leq L$ and the 
 delicate point is the behaviour of the function for $S$ large. For $S>L$, $F_{\beta,L}(S)\sim S^{\frac{2}{2\beta+1}}$, and (since $\frac{2\beta}{2\beta+1}\leq 1$) the terms  $\norm{F_{\beta,L}(S)^\beta}_{L^{\frac{2n}{n-2}}(M,g)}$ and $\norm{F_{\beta,L}(S)}_{L^{2\beta}(M,g)}$ can be controlled via $\norm{S}_{L^{\frac{2n}{n-2}}(M,g)}$ and $\norm{S}_{L^2(M,g)}$ respectively. These latter norms are bounded\footnote{This is the point where 
the estimate $xG_{\beta,L}-\frac{n}{2}H_{\beta,L} \leq f_{\beta,L}$ was necessary. Otherwise, defining 
$F_{\beta,L}$ in terms of $xG_{\beta,L}-\frac{n}{2}H_{\beta,L}$ would have caused $F_{\beta,L}(S)$ to go as $S^{\frac{3}{2\beta+1}}$ for large $L$ and we would not be able to guarantee that $\norm{F_{\beta,L}(S)^\beta}^2_{L^{\frac{2n}{n-2}}(M,g)}$ is finite.} because of $S\in C^0([0,T];H^1(M,g))$ and \eqref{eq:gSobolev}.
\medskip
 
We can compare $\norm{F_{\beta,L}(S)^\beta}^2_{L^{\frac{2n}{n-2}}(M,g)}$ and
$\norm{\phi_{\beta,L}(S)}_{L^{\frac{2n}{n-2}}(M,g)}^2$ since we have the following pointwise estimates.  
Directly from the definition we have
\begin{align*}
F_{\beta,L}(x)^\beta=\begin{dcases} \beta^{\frac{\beta}{2\beta+1}} x^\beta \leq
\beta x^\beta, & x\leq L, \\ (n\beta^2)^{\frac{\beta}{2\beta+1}} L^\beta \left(\frac{x}{L}\right)^{\frac{2\beta}{2\beta+1}}
\leq n\beta L^{\beta-1} x, & x>L.\end{dcases}
\end{align*}
Similarly, we may estimate $\phi_{\beta,L}$ from below
\begin{align*}
\phi_{\beta,L}(x)&=\begin{dcases} x^{\beta} = x^{\beta}, & x\leq L, \\ 
\beta L^{\beta-1} x -(\beta-1)L^{\beta} \geq xL^{\beta-1}, & x\leq L. \end{dcases}
\end{align*}
so combined we find $n\beta \phi_{\beta,L}(x)\geq F_{\beta,L}(x)^\beta$.
By sufficiently shrinking $\delta>0$ (choosing $\delta\leq \frac{4(n-1)}{n^3\beta A(T)}$ to be precise), we can thus ensure for all $L>0$
$$
\delta \alpha \norm{F_{\beta,L}(S)^\beta}^2_{L^{\frac{2n}{n-2}}(M,g)}\leq 
\frac{(n-1)}{A(T)} \norm{\phi_{\beta,L}(S)}_{L^{\frac{2n}{n-2}}(M,g)}^2,
$$ 
and therefore deduce from \eqref{eq:HEvolStep} and \eqref{eq:HEvolStep2}
\begin{equation}
\label{eq:OnlyFLeft}
\partial_t \norm{H_{\beta,L}(S)}_{L^1(M,g)}  \leq 
C_T \norm{H_{\beta,L}(S)}_{L^1(M,g)}+
C'_T \norm{F_{\beta,L}(S)}^{\frac{1}{1-\alpha}+2\beta}_{L^{2\beta}(M,g)},
\end{equation}
for uniform constants $C_T, C'_T>0$ where $C_T$ is given above and
\[C'_T:= \delta^{-\frac{n}{4\beta-n}}\left(\frac{4\beta-n}{4\beta}\right)\]
and $\delta\leq \frac{4(n-1)}{n^3\beta A(T)}$. The point is that both depend only on $T>0$ and $S_0$.

\noindent We then compare $F_{\beta,L}(x)^{2\beta}$ to $H_{\beta,L}(x)$ as follows. 
From the definition of $F_{\beta,L}(x)$ again we find
\begin{align*}
F_{\beta,L}(x)^{2\beta}&=\begin{dcases} \beta^{\frac{2\beta}{2\beta+1}} x^{2\beta} 
\leq \beta x^{2\beta}, & x\leq L, \\ (n\beta^2)^{\frac{2\beta}{2\beta+1}} L^{2\beta} \left(\frac{x}{L}\right)^{\frac{4\beta}{2\beta+1}}
\leq n\beta^2 L^{2(\beta-1)} x^2, & x>L.\end{dcases}
\end{align*} 
Consulting \eqref{eq:H-Manipulations} we find
\[ 4n\beta H_{\beta,L}(x)\geq \frac{2\beta}{2\beta-1} \begin{dcases} n\beta x^{2\beta}& x\leq L \\ n\beta^2 L^{2(\beta-1)} x^2 & x>L.\end{dcases}\]
We therefore conclude $4n\beta H_{\beta,L}(x)\geq F_{\beta,L}(x)^{2\beta}$.
Defining 
\[C''_T:= \max \left\{\left(4n\beta\right)^{1+\frac{2}{4\beta-n}} C'_T,C_T\right\},\]
 we deduce from \eqref{eq:OnlyFLeft} that
\[
\partial_t \norm{H_{\beta,L}(S)}_{L^1(M,g)} \leq C''_T 
\left(1+\norm{\frac{}{}H_{\beta,L}(S)}_{L^1(M,g)}^{\frac{2}{4\beta-n}}\right)\norm{H_{\beta,L}(S)}_{L^1(M,g)}.
\]
Setting $\beta=\frac{n^2}{4(n-2)}$, we can rewrite this differential inequality as 
\[
\partial_t \log \left(\norm{H_{\beta,L}(S)}_{L^1(M,g)}\right)\leq 
C''_T \left(1 + \norm{\frac{}{}H_{\beta,L}(S)}_{L^1(M,g)}^{\frac{n-2}{n}} \right).\]
Integrating this differential inequality in time, we conclude
\begin{align*}
&\log   \left(\norm{H_{\beta,L}(S(T))}_{L^1(M,g)}\right)  \\ 
& \leq \log \left(\norm{H_{\beta,L}(S_0)}_{L^1(M,g_0)}\right) + C''_T T + 
C''_T\int_0^T  \norm{H_{\beta,L}(S)}_{L^1(M,g)}^{\frac{n-2}{n}}\, dt.
\end{align*}
Taking the limit $L\to \infty$ (using Fatou's lemma and dominated convergence as before
in the final step of the proof of Proposition \ref{Prop:LowerCurvBound2} )
and using Lemma \ref{Lem1}, we deduce
\[
\log \norm{\, S(T)^{\frac{n^2}{2(n-2)}}\, }_{L^1(M,g)} \leq 
\log\norm{\, S_0^{\frac{n^2}{2(n-2)}}\, }_{L^1(M,g)}+C''_T T + C''_TC(T),
\]
which proves the statement for $n\geq 4$. \medskip

\noindent The above proof would almost work for $n=3$. The problem is that $xG_{\beta,L}-\frac{n}{2}H_{\beta,L}\leq f_{\beta,L}$ no longer holds true, and one would have a problem showing that the norms in \eqref{eq:HEvolStep2} are finite. One solution is to redefine the approximation functions $\phi_{\beta,L}$, $G_{\beta,L}$, and $H_{\beta,L}$ to ensure $xG_{\beta,L}(x)-\frac{n}{2}H_{\beta,L}(x)$ is dominated by a function $f_{\beta,L}$ going like at most $ x$ for large $x$, rather than $x^2$. This is a non-trivial task, because it is also important for the above argument that one can find constants (depending on $n,\beta$, but not $L$) such that $C_1 H_{\beta,L}(x)\geq  \phi_{\beta,L}(x)^2$, $C_2 H_{\beta,L}(x)\geq xG_{\beta,L}(x)$, $C_3 F_{\beta,L}(x)^{\beta}\leq \phi_{\beta,L}(x)$ and $C_4 H_{\beta,L}(x)\geq  F_{\beta,L}(x)^{2\beta}$, where $F_{\beta,L}(x)=(xf_{\beta,L}(x))^{\frac{1}{2\beta+1}}$. Consider the following family of approximation functions, $\nu\leq 1$ and $\nu\notin\left\{0, \frac{1}{2}\right\}$.
\begin{equation}
\tilde{\phi}_{\beta,L}(x):= \begin{dcases} x^{\beta}, & x\leq L \\
 \frac{\beta}{\nu} L^{\beta-\nu}x^{\nu} +L^\beta\left(1-\frac{\beta}{\nu}\right), & x>L.
 \end{dcases}
 \end{equation}
\begin{equation}
\tilde{G}_{\beta,L}(x):= \int_0^x \tilde{\phi}_{\beta,L}'(y)^2\, dy=\begin{dcases} \frac{\beta^2}{2\beta-1} x^{2\beta-1},  & x\geq L,\\
\frac{\beta^2L^{2(\beta-\nu)}}{2\nu-1}x^{2\nu-1} -\frac{2\beta^2 L^{2\beta-1}(\beta-\nu)}{(2\nu-1)(2\beta-1)},  & x>L,
\end{dcases}
\end{equation}
and
\begin{align*}
\tilde{H}_{\beta,L}(x):= \int_0^x \tilde{G}_{\beta,L}(y)\, dy
=\begin{dcases} \frac{\beta}{2(2\beta-1)} x^{2\beta}, \! & \! x\leq L,\\ 
\frac{\beta^2L^{2(\beta-\nu)}}{2\nu(2\nu-1)}x^{2\nu} -\frac{2\beta^2 L^{2\beta-1} 
(\beta-\nu)}{(2\nu-1)(2\beta-1)} x -C_{\beta,\nu}L^{2\beta}, \! & \! x>L,
\end{dcases}
\end{align*}
where 
\[
C_{\beta,\nu}:= \frac{\beta\left(\beta(2\beta-1)+4\nu\beta(\nu-\beta)+\nu(1-2\nu) 
\right)}{2\nu(2\beta-1)(2\nu-1)}.
\]
In the $n\geq 4$ case we considered these functions with $\nu=1$.
These functions have the same qualitative properties of as before, namely that $\tilde{\phi}_{\beta,L}\xrightarrow{L\to \infty} x^\beta$ and $\tilde{\phi}_{\beta,L}\in C^1(\R_+)$ with $\tilde{\phi}_{\beta,L}'\in L^{\infty}(\R_+)$, and so on for $\tilde{G}_{\beta,L}$ and $\tilde{H}_{\beta,L}$. We can therefore use $\tilde{G}_{\beta,L}(S)$ as test function in \eqref{eq:ScalarEvol} and deduce the analogue of \eqref{eq:CurvatureEvolStep}, namely
\begin{align}
&\partial_t \int_M \tilde{H}_{\beta,L}(S)\, d\Vol_g +(n-1)\int_M \vert \nabla \tilde{\phi}_{\beta,L}(S)\vert^2\, d\Vol_g \notag \\ 
&\leq \int_M (S-\rho)\left(S\tilde{G}_{\beta,L}(S)-\frac{n}{2} \tilde{H}_{\beta,L}(S)\right)\, d\Vol_g.
\label{eq:TildeCurvatureEvolStep}
\end{align}
Consider the the expression $x\tilde{G}_{\beta,L}(x)-\frac{n}{2}\tilde{H}_{\beta,L}(x)$ for $x>L$,
\begin{align*}
&x\tilde{G}_{\beta,L}(x)-\frac{n}{2}\tilde{H}_{\beta,L}(x)\\
&=\frac{\beta^2 L^{2(\beta-\nu)}}{2\nu (2\nu-1)}\left(2\nu-\frac{n}{2}\right)x^{2\nu} +\frac{2\beta^2 (\beta-\nu)L^{2\beta-1}}{(2\beta-1)(2\nu-1)}\left(\frac{n}{2}-1\right)x+\frac{n}{2}C_{\beta,\nu}L^{2\beta}.
\end{align*}
From this one sees that when $0<\nu \leq \frac{n}{4}$ and $\beta \geq \frac{n}{4}$, the first two terms becomes negative. 
So assume from now on that $0<\nu \leq \frac{n}{4}$ and later we will make a choice of $\beta \geq \frac{n}{4}$. 
Introduce the function 
\[\tilde{f}_{\beta,L}(x):= \begin{cases} \beta x^{2\beta} & x\leq L\\ \frac{n}{2}\vert C_{\beta,\nu}\vert L^{2\beta} & x>L\end{cases},\]
which has the property that $x\tilde{G}-\frac{n}{2}H_{\beta,L}\leq \tilde{f}_{\beta,L}(x)$ for all $x\geq 0$ and $L>0$,
as long as $\beta \geq \frac{n}{4} \geq \nu$. Proceeding exactly as in the $n\geq 4$ case we deduce
\begin{equation}\label{TildeH-estimate2}
\begin{split}
&\partial_t \norm{\tilde{H}_{\beta,L}(S)}_{L^1(M,g)} \\ 
&\leq (n-1)\frac{B(T)}{A(T)} \norm{\tilde{\phi}_{\beta,L}(S)^2}_{L^1(M,g)}
-\frac{(n-1)}{A(T)} \norm{\tilde{\phi}_{\beta,L}(S)}_{L^{\frac{2n}{n-2}}(M,g)}^2 \\
&+\rho(0)\norm{S\tilde{G}_{\beta,L}(S)-\frac{n}{2}\tilde{H}_{\beta,L}(S)}_{L^1(M,g)}
+\norm{S\tilde{f}_{\beta,L}(S)}_{L^1(M,g)}.
\end{split}\end{equation}
We compare $\tilde{\phi}_{\beta,L}(x)^2$, $x\tilde{G}_{\beta,L}(x)$ and $\tilde{H}_{\beta,L}(x)$ as in 
as for \eqref{eq:H-Comparison}, and conclude by similar arguments for all $\beta\geq 1, L>0$, 
and some $L$-independent constants $C_1$, $C_2$
\begin{equation}
C_1 \tilde{H}_{\beta,L}(x)\geq \tilde{\phi}_{\beta,L}(x)^2, \quad C_2 \tilde{H}_{\beta,L}(x)\geq x\tilde{G}_{\beta,L}(x). 
\label{eq:TildeComparison}
\end{equation}
We now proceed as before, getting
\begin{equation}\label{eq:TildeHEvolStep}
\begin{split}
\partial_t \norm{\tilde{H}_{\beta,L}(S)}_{L^1(M,g)}  &\leq C_T \norm{\tilde{H}_{\beta,L}(S)}_{L^1(M,g)} 
\\ &+ \norm{S\tilde{f}_{\beta,L}(S)}_{L^1(M,g)} -\frac{(n-1)}{A(T)} \norm{\tilde{\phi}_{\beta,L}(S)}_{L^{\frac{2n}{n-2}}(M,g)}^2,
\end{split}\end{equation} 
where the uniform constant $C_T>0$ is explicitly given by 
$$
C_T := C_1(n-1) \, \frac{A(T)}{B(T)} + \rho(0)\left(\frac{n}{2}+C_2\right).
$$ 
Introduce the non-negative, real function $\tilde{F}_{\beta,L}$ via 
$$
\tilde{F}_{\beta,L}(x) := \left( x\tilde{f}_{\beta,L}(x) \right)^{\frac{1}{2\beta+1}}.
$$
Assume $\beta>\frac{n}{4}$, which holds e.g. for $\beta=\frac{n^2}{4(n-2)}$. 
Set $\alpha:= \frac{n}{4\beta}<1$ and choose any $\delta>0$. Observe that by the 
H\"older inequality in the first estimate and the Young inequality in the second, we obtain
\begin{equation}\label{eq:TildeHEvolStep2}
\begin{split}
&\norm{\tilde{F}_{\beta,L}(S)^{2\beta+1}}_{L^1(M,g)} 
 \leq \norm{\frac{}{}\tilde{F}_{\beta,L}(S)}^{2\alpha \beta}_{L^{\frac{2n\beta}{n-2}}(M,g)} 
 \norm{ \frac{}{}\tilde{F}_{\beta,L}(S)}^{1+2(1-\alpha)\beta}_{L^{2\beta}(M,g)}  \\
 &\leq \delta \alpha \norm{\frac{}{}\tilde{F}_{\beta,L}(S)^\beta}^2_{L^{\frac{2n}{n-2}}(M,g)} 
 +\delta^{-\frac{\alpha}{1-\alpha}}(1-\alpha) \norm{\frac{}{}\tilde{F}_{\beta,L}(S)}^{\frac{1}{1-\alpha}+2\beta}_{L^{2\beta}(M,g)}.
\end{split}\end{equation} 
These integrals are finite for the same reasons as in the $n\geq 4$ case. \medskip

\noindent We shall from now on set set $\nu=\frac{\beta}{2\beta+1}$ and $\beta=\frac{n^2}{4(n-2)}$, which translates into $\nu=\frac{9}{22}$ for $n=3$. Notice that this choice satisfies $v\leq \frac{n}{4}$, so the manipulations up until now are allowed. The reason for choosing this $\nu$ is that then
\begin{align*}
\tilde{F}_{\beta,L}(x)^{\beta}=\begin{dcases} \beta^{\nu} x^\beta, & x\leq L, \\ 
\left( \frac{n}{2} \, \vert C_{\beta,\nu}\vert\right)^{\nu} L^{\beta-\nu}x^{\nu}, & x>L.
\end{dcases}
\end{align*}
This is easily comparable to $\tilde{\phi}_{\beta,L}(x)$. Since 
\[\tilde{\phi}_{\beta,L}(x)\geq \frac{\beta}{\nu}L^{\beta-\nu} x^\nu\]
for $x>L$, we see that defining 
\[C_3^{-1}\coloneqq \max\left\{ \beta^{\nu}, \frac{\nu}{\beta}\left(\frac{n}{2}\vert C_{\beta,\nu}\vert\right)^{\nu}\right\}\]
we achieve\footnote{This is a somewhat delicate point. If one chooses $\nu$ small, it is easy to make $x\tilde{G}-\frac{n}{2}\tilde{H}$ sublinear, but if $\nu$ is too small, $\tilde{F}$ will increase faster than $\tilde{\phi}$, ruining the comparison. On the other hand, if $\nu$ is bigger than $\frac{n}{4}$ we see above that $x\tilde{G}-\frac{n}{2}\tilde{H}$ becomes too large to guarantee the finiteness of the integrals in \eqref{eq:TildeHEvolStep2}.}  $C_3 \tilde{F}_{\beta,L}(x)^\beta \leq \tilde{\phi}_{\beta,L}(x).$
So if we choose $\delta\leq \frac{(n-1)}{A(T)} \frac{4C_3^2 \beta}{n}$, then $\delta \alpha \norm{\frac{}{}\tilde{F}_{\beta,L}(S)^\beta}^2_{L^{\frac{2n}{n-2}}(M,g)}-\frac{(n-1)}{A(T)} \norm{\tilde{\phi}_{\beta,L}(S)}_{L^{\frac{2n}{n-2}}(M,g)}^2\leq 0$ holds for all $L>0$, and we deduce from \eqref{eq:TildeHEvolStep} and \eqref{eq:TildeHEvolStep2} that 
\begin{equation}
\label{eq:TildeOnlyFLeft}
\partial_t \norm{\tilde{H}_{\beta,L}(S)}_{L^1(M,g)}  \leq 
C_T \norm{\tilde{H}_{\beta,L}(S)}_{L^1(M,g)}+
C'_T \norm{\tilde{F}_{\beta,L}(S)}^{\frac{1}{1-\alpha}+2\beta}_{L^{2\beta}(M,g)},
\end{equation}
for uniform constants $C_T, C'_T>0$ where $C_T$ is given above and
\[C'_T:= \delta^{-\frac{n}{4\beta-n}}\left(\frac{4\beta-n}{4\beta}\right)\]
and $\delta\leq \frac{(n-1)}{A(T)} \frac{4C_3^2 \beta}{n}$. The point is that both depend only on $T>0$ and $S_0$. The final comparison we need is that $C_4 H_{\beta,L}(x)\geq  F_{\beta,L}(x)^{2\beta}$ holds for some $C_4>0$ independent of $L$, and here is a way to see that this is doable. For $x\leq L$ both functions are proportional, so there is nothing to show. Inserting $\nu=\frac{\beta}{2\beta+1}$ into the definition of $\tilde{H}_{\beta,L}(x)$ yields (for $x>L$)
\[\tilde{H}_{\beta,L}(x)=L^{2\beta}\left( \frac{4\beta^4}{2\beta-1} \left(\frac{x}{L}\right) -\frac{\beta(2\beta+1)^2}{2} \left(\frac{x}{L}\right)^{2\nu} +\beta^2\right),\]
which shows that $\tilde{H}_{\beta,L}(x)$ is dominated by a positive linear term for $x>L$, which will dominate the sublinear term $x^{2\nu}$ of $\tilde{F}_{\beta,L}(x)^{2\beta}$.
Defining 
\[C''_T:= \max \left\{C_4^{1+\frac{2}{4\beta-n}} C'_T,C_T\right\},\]
 we deduce from \eqref{eq:TildeOnlyFLeft} that
\[
\partial_t \norm{\tilde{H}_{\beta,L}(S)}_{L^1(M,g)} \leq C''_T 
\left(1+\norm{\frac{}{}\tilde{H}_{\beta,L}(S)}_{L^1(M,g)}^{\frac{2}{4\beta-n}}\right)\norm{\tilde{H}_{\beta,L}(S)}_{L^1(M,g)},
\]
The rest of the proof then runs exactly as in the $n\geq 4$ case, giving us our required bound also for $n=3$.
\end{proof}

This completes the first step on the way to Theorem \ref{Prop:CurvatureBound}, 
proving a uniform $L^{\frac{n^2}{2(n-2)}}(M,g)$-norm bound on $S$. Before we can 
go on to prove Theorem \ref{Prop:CurvatureBound} by a Moser iteration argument, 
we need the following parabolic Sobolev inequality.

\begin{Lem}
 Let $A(T)$ and $B(T)$ denote the constants of the (elliptic) Sobolev 
 inequality \eqref{eq:gSobolev}. Then for any $f\in C^0([0,T];H^1(M,g))$ we have
(writing $M_T := M \times [0,T]$\footnote{We write $L^p(M_T,g) \equiv L^p(M_T, g\oplus dt^2)$ for any $p\geq 1$.})
\begin{equation} \label{eq:ParabolicSob}
\begin{split}
\norm {f^2 }_{L^\frac{n+2}{n}(M_T,g)}& \leq
\frac{n}{n+2}\left( A(T)\norm{\nabla f}_{L^2(M_T,g)}^2+B(T) \norm{f}^2_{L^2(M_T,g)}\right) \\
 &+ \frac{2}{n+2} \sup_{t \in [0,T]} \norm{f(t)}^2_{L^2(M,g)}.
\end{split}
\end{equation}
\end{Lem}

\begin{proof}
The statement and the proof are close to \cite[Eqn. 12]{MCZ12}. We compute
\begin{align*} 
\int_0^T \int_M f^{\frac{2(n+2)}{n}}\, d\Vol_g\, dt&=\int_0^T \int_M f^2 f^{\frac{4}{n}} \, d\Vol_g \, dt \\ 
&\leq \int_0^T \left( \norm{f}_{L^{\frac{2n}{n-2}}(M,g)}^2 \norm{f}^{\frac{4}{n}}_{L^2(M,g)}\right)\, dt \\ 
&\leq \int_0^T \left(A(T)\norm{\nabla f}^2_{L^2(M,g)}+B(T)\norm{f}^2_{L^2(M,g)}\right)
\left(  \norm{f}^{\frac{4}{n}}_{L^2(M,g)}\right) \, dt \\ 
& \leq  \left(A(T)\norm{\nabla f}^2_{L^2(M_T,g)}+B(T)\norm{f}^2_{L^2(M_T,g)}\right)
\sup_{t \in [0,T]} \left(  \norm{f}^{\frac{4}{n}}_{L^2(M,g)}\right).
\end{align*}
where in the first estimate we applied the H\"older inequality with $p=n/2$ and $q=n/(n-2)$,
and in the second estimate applied \eqref{eq:gSobolev}.
Raising both sides of the inequality to $\frac{n}{n+2}$ and using 
Young's inequality $AB\leq \frac{A^p}{p}+\frac{B^q}{q}$ with $p=\frac{n+2}{n}$ and $q=\frac{n+2}{2}$ we 
arrive at the estimate as claimed.
\end{proof} \ \\[-6mm]

\noindent We are now ready to tackle Theorem \ref{Prop:CurvatureBound}.

\begin{proof}[Proof of Theorem \ref{Prop:CurvatureBound}]
Since we assume that $(S_0)_- \in L^\infty(M)$, we have uniform bounds on $S_-$ by 
Proposition \ref{Prop:LowerCurvBound2}. Thus, it suffices to prove the statement for $S_+$.
Therefore, we may replace $S$ by $S_+$, replacing the evolution equation \eqref{eq:ScalarEvol} for $S$  
by the inequality \eqref{eq:S+Evol} for $S_+$. Hence we continue under the assumption
$S\equiv S_+ \geq 0$, subject to \eqref{eq:S+Evol}.  \medskip

Let $\eta \in C^1([0,T],\R_+)$ be non-decreasing with $\eta(0)=0$ and $\| \eta \|_\infty \leq 1$. 
We want to use $\beta^2\eta^2 S^{2\beta-1}/ (2\beta -1)$ (with $\beta>1$) as a test function in the weak formulation of \eqref{eq:S+Evol}. The problem is of course that the chain rule fails to hold in general, 
so we use the same workaround as in Proposition \ref{Prop:LowerCurvBound2} and 
Corollary \ref{Lem:UnifCurvBound}. Let $L>0$ and define $\phi_{\beta,L},G_{\beta,L},$ and 
$H_{\beta,L}$ as before. Using $\eta(s)^2 G_{\beta,L}(S)$ as test function in \eqref{eq:S+Evol} we get
\[
\int_M (\partial_s S)\eta^2 G_{\beta,L}(S)\, d\Vol_g +
(n-1)\int_M \eta^2 \vert \nabla \phi_{\beta,L}(S)\vert^2\, d\Vol_g \leq \int_M S G_{\beta,L}(S) \vert S-\rho\vert\, d\Vol_g.
\]
On the right hand side we observe (by a direct computation) that $SG_{\beta,L}(S)\leq \frac{\beta^2}{2\beta-1}\phi_{\beta,L}(S)^2$.  We integrate this  in time for any $t \in [0,T]$
\begin{equation}\label{auxiliary3}
\begin{split}
\int_0^t \int_M (\partial_s S)\eta^2 G_{\beta,L}(S)\, d\Vol_g\, ds &+(n-1)\int_0^t\int_M \eta^2 \vert \nabla \phi_{\beta,L}(S)\vert^2\, d\Vol_g\, ds \\ & \leq \frac{\beta^2}{2\beta-1}\int_0^t \int_M \eta^2  \phi_{\beta,L}(S)^2 \vert S-\rho\vert\, d\Vol_g\, ds.
\end{split}\end{equation}
We rewrite the first term on the left hand side of \eqref{auxiliary3} using \eqref{eq:DVol} 
\begin{align*}
 \int_0^t \int_M \eta^2 (\partial_s S) &G_{\beta,L}(S)\, d\Vol_g\, ds
\equiv \int_0^t \int_M \eta^2 \partial_s H_{\beta,L}(S) \, d\Vol_g\, ds  \\  
&= \int_M \eta^2 H_{\beta,L}(S)\, d\Vol_g (s=t) - 2 \int_0^t\int_M \eta \, \dot{\eta} \, H_{\beta,L}(S) \, d\Vol_g\, ds
\\ &+\frac{n}{2} \int_0^t \int_M \eta^2 H_{\beta,L}(S) (S-\rho)\, d\Vol_g\, ds,
\end{align*}
 where we write $\dot{\eta} \equiv \partial_s \eta$ and used $\eta(0)=0$. 
 Plugging this into \eqref{auxiliary3}, we obtain
 \begin{align*}
 \int_M \eta^2 H_{\beta,L}(S)\, &d\Vol_g (s=t) +
 (n-1)\int_0^t\int_M \eta^2 \vert \nabla \phi_{\beta,L}(S)\vert^2\, d\Vol_g\, ds  \\
 & \leq \int_0^t \int_M \eta^2 \left(  \frac{\beta^2}{2\beta-1}\phi_{\beta,L}(S)^2 +\frac{n}{2}H_{\beta_L}(S)\right)\vert S-\rho\vert\, d\Vol_g\, ds \\ &+ 2 \int_0^t\int_M \eta \, \dot{\eta} \, H_{\beta,L}(S) \, d\Vol_g\, ds.
 \end{align*} 
We now take the supremum over $t\in [0,T]$ and appeal to the parabolic Sobolev inequality 
\eqref{eq:ParabolicSob} with $f=\eta \phi_{\beta,L}(S)$ . The result is
\begin{equation}\label{auxiliary4} 
\begin{split} 
 & \frac{(n-1)}{nA(T)} \left( (n+2) \norm{\eta^2 \phi_{\beta,L}(S)^2}_{L^{\frac{n+2}{n}}(M_T,g)} 
 - 2 \sup_{t \in [0,T]}\norm{\eta \phi_{\beta,L}(S)}^2_{L^2(M,g)} \right. \\  & \left. -
 n B(T) \norm{\eta \phi_{\beta,L}(S)}_{L^2(M_T,g)}^2 \right)
  +\sup_{t \in [0,T]} \int_M \eta^2 H_{\beta,L}(S)\, d\Vol_g 
\\ &\leq \int_0^T \int_M \eta^2 \left( \frac{\beta^2}{2\beta-1}\phi_{\beta,L}(S)^2 
 +\frac{n}{2}H_{\beta_L}(S)\right)\vert S-\rho\vert\, d\Vol_g\, dt
 \\ & + 2 \int_0^T\int_M \eta \, \dot{\eta} \, H_{\beta,L}(S) \, d\Vol_g\, dt.
\end{split} 
\end{equation}
 By increasing $A(T)>0$ if needed, we may assume (also noting that $H_{\beta,L}$
 and $\phi^2_{\beta,L}$ are comparable by \eqref{eq:H-Comparison})
 \[
 \sup_{t \in [0,T]}\int_M \eta^2 H_{\beta,L}(S)\, d\Vol_g  -
 \frac{2(n-1)}{nA(T)} \sup_{t \in [0,T]} \norm{\eta \phi_{\beta,L}(S)}^2_{L^2(M,g) } \geq 0
 \]
for all $\beta\geq 1$ and $L>0$. We may therefore drop these terms from \eqref{auxiliary4}.
Taking the limit $L\to \infty$ (using Fatou's lemma and the dominated convergence theorem) we get
 \begin{align*}
 & (n-1)\frac{n+2}{nA(T)} \norm{\eta^2 S^{2\beta}}_{L^{\frac{n+2}{n}}(M_T)}  -\frac{B(T)(n-1)}{A(T)}\norm{\eta S^{\beta}}_{L^2(M_T,g)}^2   \\
 &\leq \left( \frac{\beta^2}{2\beta-1}+\frac{\beta}{4n(2\beta-1)}\right) \int_0^T \int_M \eta^2 S^{2\beta}\vert S-\rho\vert\, d\Vol_g\, dt
 \\ & +\frac{\beta}{2\beta-1} \int_0^T\int_M \eta \dot{\eta} S^{2\beta} \, d\Vol_g\, dt.
 \end{align*} 
 Introducing $C:= \frac{nA(T)}{(n+2)(n-1)}$ we get from this inequality for any $\beta>1$
 \begin{align*}
\norm{\eta^2 S^{2\beta}}_{L^{\frac{n+2}{n}}(M_T)} 
&\leq \frac{nB(T)}{n+2} \norm{\eta S^{\beta}}^2_{L^2(M_T)} +C \int_0^T\int_M \eta \dot{\eta} S^{2\beta} \, d\Vol_g\, dt
\\ &+2 C\beta \int_0^T \int_M \eta^2 S^{2\beta}\vert S-\rho\vert\, d\Vol_g\, dt.
 \end{align*}
We apply the Hölder inequality with $p=\frac{n^2}{2(n-2)}$ to the last integral on the 
right hand side of the last inequality. Using Corollary \ref{Lem:UnifCurvBound} to get a bound on 
the integral of $\vert S-\rho\vert^p$, we conclude
\begin{equation} \label{eq:MoserStart}
\begin{split}
\norm{\eta^2 S^{2\beta}}_{L^{\frac{n+2}{n}}(M_T)} 
&\leq \frac{nB(T)}{n+2} \norm{\eta S^{\beta}}^2_{L^2(M_T)}  \\ &+C \int_0^T\int_M \eta \dot{\eta} S^{2\beta} \, d\Vol_g\, dt+ C(T)\beta \norm{ \eta^2 S^{2\beta}}_{L^{N}(M_T)},
\end{split} \end{equation}
with $N:= \frac{p}{p-1}=\frac{n^2}{n^2-2n+4}<\frac{n+2}{n}$. This is almost the expression we want to iterate, but the presence of $\dot{\eta}$ means we have to shrink our time interval in the iteration (as is standard for parabolic Moser iteration). Here are the details (inspired by \cite[pp. 889-890]{MCZ12}). \medskip

\noindent Consider the sequence $t_k:= \left(\frac{1}{2}-\frac{1}{2^k}\right)T$ for integers $k\geq 1$.  Let $M_k:= M\times [t_k,T]$,  $M_1=M_T$ and $M_\infty =M\times \left [\frac{T}{2},T\right]$. Choose test functions 
$\eta_k \in C^1([0,T],\R_+)$, non-decreasing with $\| \eta_k \|_\infty \leq 1$, such that
\[
\eta_k(t)=\begin{cases} 0, & t\leq t_{k-1}, \\ 1, & t\geq t_k.\end{cases}
\]
The choice of $\{\eta_k\}_k$ can be made subject to a bound on the derivative $0\leq \dot{\eta}_k \leq 2^{k+1}/T$,
which we henceforth assume. Using these functions in \eqref{eq:MoserStart}, we find
\begin{equation} 
\label{eq:MoserStep}
\begin{split}
&\norm{ S^{2\beta}}_{L^{\frac{n+2}{n}}(M_k)}= \norm{\eta_k^2 S^{2\beta}}_{L^{\frac{n+2}{n}}(M_k)}
\leq \norm{\eta_k^2 S^{2\beta}}_{L^{\frac{n+2}{n}}(M_T)} \\ 
&\leq \frac{nB(T)}{n+2} \norm{\eta_k S^{\beta}}^2_{L^2(M_T)} 
+C \int_0^T\int_M \eta_k \dot{\eta}_k S^{2\beta} \, d\Vol_g\, dt+ C(T)\beta \norm{ \eta_k^2 S^{2\beta}}_{L^{N}(M_T)} \\
&\leq\tilde{C}(T) \beta 2^{k+1} \norm{S^{2\beta}}_{L^N(M_{k-1})},
\end{split}\end{equation}
where the second inequality uses \eqref{eq:MoserStart}
and last step uses $\dot{\eta}\leq 2^{k+1}/T$ together with the H\"older inequality to compare
$L^1$ and $L^N$ norms. This is the equation we will be iterating. Introduce $\gamma:= 2\beta N$ and $\rho:= \frac{n+2}{nN}=\frac{n^3+8}{n^3}>1$. Then \eqref{eq:MoserStep} reads
\[\norm{S}_{L^{\rho \gamma}(M_k)}\leq \left(\tilde{C}(T) \gamma 2^k\right)^{\frac{N}{\gamma}} \norm{S}_{L^{\gamma}(M_{k-1})}.\]
Replacing $\gamma$ by $\rho^m\gamma$ for $m\geq 0$ results in
\[\norm{S}_{L^{\rho^{m+1} \gamma}(M_{k+m})}\leq \left(\tilde{C}(T) \rho^m\gamma 2^{k+m}\right)^{\frac{N}{\rho^m\gamma}} \norm{S}_{L^{\rho^m\gamma}(M_{k+m-1})},\]
which can be iterated down to
\[\norm{S}_{L^{\rho^{m+1} \gamma}(M_{k+m})}\leq \prod_{i=0}^m \left(\tilde{C}(T) \rho^i \gamma 2^{k+i}\right)^{\frac{N}{\rho^i \gamma}} \norm{S}_{L^{\gamma}(M_{k-1})}.\]
The expression $\prod_{i=0}^m \left(\tilde{C}(T) \rho^i \gamma 2^{k+i}\right)^{\frac{N}{\rho^i \gamma}}$ converges as $m\to \infty$, as one checks by computing the logarithm
\begin{align*}
&\lim_{m\to \infty}\log \prod_{i=0}^m \left(\tilde{C}(T) \rho^i \gamma 2^{k+i}\right)^{\frac{N}{\rho^i \gamma}}=\frac{N}{\gamma} \sum_{i=0}^\infty \left(\log(2^k \tilde{C}(T) \gamma) \frac{1}{\rho^i} +\log(2\rho) \frac{i}{\rho^i} \right)\\
&=\frac{N}{\gamma}\left(\frac{\rho}{\rho-1} \log(\tilde{C}(T)\gamma 2^k) +\log(2\rho)\frac{\rho}{(\rho-1)^2}\right).
\end{align*}
We therefore deduce for some uniform constant $C_T>0$
\[ \norm{S}_{L^{\infty}\left(M\times\left[\frac{T}{2},T\right] \right)} \leq \lim_{m\to \infty} \norm{S}_{L^{\rho^{m+1} \gamma}(M_{k+m})}\leq C_T \norm{S}_{L^{\gamma}(M_{k-1})} \leq C_T \norm{S}_{L^{\gamma}(M_T)}.\]
Choosing $\beta=\frac{n^2-2n+4}{4(n-2)}\iff \gamma =\frac{n^2}{2(n-2)}$, we can estimate the right hand side using Corollary \ref{Lem:UnifCurvBound} and deduce for some uniform constant $C(T)>0$
\[ \norm{S}_{L^{\infty}\left(M\times\left[\frac{T}{2},T\right] \right)}\leq C(T).\]

\end{proof}

\begin{Rem}
It is worth pointing out that we do not assume $S_0\in L^{\infty}(M)$, only that $S_0\in L^{\frac{n^2}{2(n-2)}}(M)$. The above proof tells us that $S\in L^{\infty}(M)$ for positive times, even if the initial curvature is unbounded. This is analogous to the well known behaviour of the heat equation, where the solutions for positive times are often much more regular than the initial data. 
\end{Rem}

\section{Long-time existence of the normalized Yamabe flow}\label{long-section}

\noindent We can now establish our main Theorem \ref{main-thm}, which explicitly reads as follows.

\begin{Thm}\label{main-thm2}
Let $(M,g_0)$ be a Riemannian manifold of dimension $n= \dim M \geq 3$, 
such that the following four assumptions hold:
\begin{enumerate}
\item The Yamabe constant $Y(M,g_0)$ is positive, i.e. Assumption \ref{Y-assump} holds;
\item $(M,g_0)$ is admissible, i.e. Assumption \ref{admissible-assump} holds;
\item Parabolic Schauder estimates hold on $(M,g_0)$, i.e. Assumption \ref{Schauder-assump} holds; 
\item $S_0 \in C^{1,\alpha}(M)$, i.e. Assumption \ref{S-bound} holds. \\
Moreover we require $S_0\in L^{\frac{n^2}{2(n-2)}}(M)$, and its negative part $(S_0)_- \in L^{\infty}(M)$.
\end{enumerate}
Under these assumptions, the normalized Yamabe flow $u^{\frac{4}{n-2}}g_0$ exists with 
$u \in C^{3,\alpha}(M\times [0,\infty))$ with infinite existence time and scalar curvature $S(t) \in L^\infty(M)$ 
for all $t>0$.
\end{Thm}

\begin{proof}
Short time existence of the flow with $u \in C^{3,\alpha}(M\times [0,T'])$ for some small $T'>0$ is due to Theorem \ref{short}.
Let $T>0$ be the maximal existence time, so that $u \in C^{3,\alpha}(M\times [0,T))$ without locally uniform 
control of the H\"older norms in $[0,T)$, but with no uniform control of the norms up to $t=T$. 
If $T=\infty$, there is nothing to prove. Otherwise, we proceed as follows. \medskip

Proposition \ref{Prop:LowerCurvBound2} yields a uniform (i.e. depending only on $S_0$ and the finite $T$)
lower bound on the scalar curvature $S$. 
Proposition \ref{Prop:Upper-u-Bound} and Theorem \ref{Prop:Lower-u-Bound} 
yield uniform upper and lower bounds on the solution $u$, so that $u \in L^\infty(M_T)$.
This in turn gives us bounds on the Sobolev constants 
$A(T),B(T)$ (Corollary \ref{Cor:Sob}), so we use Theorem \ref{Prop:CurvatureBound} to argue that 
$S\in L^{\infty}(M_T)$. By the evolution equation 
\[\partial_t u=-\frac{4}{n-2}(S-\rho)u,\]
we deduce $\partial_t u \in L^\infty(M_T)$. Then, arguing exactly as in \cite[Proposition 2.8]{LongTime}, 
we may then restart the flow and extend the solution past $T$. For the purpose of self-containment, 
we provide the argument here. \medskip

\noindent Let us consider the linearized equation \eqref{lin} with $u=1+v$
\begin{equation}
\partial_t v -(n-1) \Delta_0 v = - \frac{n-2}{4} S_0 + \Phi (v), \quad v(0)=0,
\end{equation}
where $\Phi(v) \in L^\infty(M_T)$, since $u,\partial_t u,\rho \in L^\infty(M_T)$.
By the third mapping property in \eqref{mapping-heat}, we conclude that
$v \in C^{1,\alpha}(M\times [0,T])$\footnote{Note that we have 
uniform control of the $C^{1,\alpha}$ norm up to $t=T$ now.}.
Rewrite flow equation \eqref{eq:YamabeFlow} using $N=\frac{n+2}{n-2}$ as follows
\begin{align}\label{flow-3}
\partial_t u - (n-1) u^{1-N} \Delta_0 u = \frac{n-2}{4} \left( \rho \, u - S_0 u^{2-N}\right). 
\end{align}
Treat the right hand side of this equation as a fixed element of $C^{0,\alpha}(M\times [0,T])$.
Since  $u^{1-N} \in C^{1,\alpha}(M\times [0,T])$ is positive and uniformly 
bounded away from zero, we may apply \eqref{Q} and \eqref{R} to 
obtain a solution $u' \in C^{2,\alpha}(M\times [0,T])$ with initial condition $u'(0)=1$.
\medskip

Note that $w\coloneqq u-u'$ solves $\partial_t w - (n-1) u^{1-N} \Delta_0 w = 0$ with zero initial condition.
By the weak maximum principle \eqref{weak-max}, $\partial_t w_{\max} \leq 0$ and $\partial_t w_{\min} \geq 0$. Due to the 
initial condition $w(0)=0$, we deduce $w\equiv 0$ and hence $u=u' \in C^{2,\alpha}(M \times [0,T])$. 
Thus $u' \in C^{2,\alpha}(M\times [0,T])$ extends $u(t)$ up to $t=T$, and we conclude
$$
u \in C^{2,\alpha}(M\times [0,T]).
$$
By the second statement of Theorem \ref{short}, we even have $u \in C^{3,\alpha}(M\times [0,T])$ 
and can now restart the flow as follows. Consider $u_0= u(T) \in C^{3,\alpha}(M)$ as the initial condition 
for the normalized Yamabe flow. By \eqref{mapping-heat2}, $e^{t\Delta_0}u_0 \in C^{3,\alpha}(M\times [0,T])$, 
where the heat operator acts without convolution in time. \medskip

We write $u=f + e^{t\Delta} u_0$ and plug this into the Yamabe flow equation \eqref{eq:YamabeFlow}
with rescaled time $\tau= (t-T)$. This yields an equation for $f$ 
\begin{align}
\left[ \partial_t - (n-1) (e^{t\Delta_0}u_0)^{1-N}\Delta_0\right] f=
Q_1(f) + Q_2(f,\partial_t f), \quad u'(0)=0,
\end{align}
where $Q_1$ and $Q_2$ denotes linear and quadratic combinations of the 
elements in brackets, respectively, with coefficients given by polynomials
in $e^{t\Delta_0}u_0$ and $\partial_t e^{t\Delta_0}u_0, \Delta_0 e^{t\Delta_0} u_0$.
Since these coefficients are of higher H\"older regularity $C^{1,\alpha}(M)$, we may set
up a contraction mapping argument in $C^{3,\alpha}$ and thus extend $u$ past the maximal 
existence time $T$ ad verbatim to the proof of Theorem \ref{short}.
This proves long-time existence. 
\end{proof}

\begin{Cor}
In the setting of the above theorem, we have
\[\lim_{t\to \infty} \int_M (S-\rho)^2\, d\Vol_g =0\]
and there exists $u_{\infty}\in L^2(M)$ such that 
\[\lim_{t\to \infty} \int_M (u-u_{\infty})^2\, d\mu=0.\]  
\end{Cor}
\begin{proof}
By \eqref{eq:rhoEvol} we have
\[\partial_t \rho =-\frac{n-2}{2} \int_M (S-\rho)^2\, d\Vol_g.\]
This shows that $\rho(t)$ is monotonously decreasing, and we know it's bounded from below by $Y(M,g_0)>0$, so $\lim_{t\to \infty} \rho(t)$ exists. Thus $\int_0^\infty \partial_t \rho(t) dt < \infty$ and thus $\partial_t \rho(t)$ must converge
to zero as $t\to \infty$. This gives the conclusion on $ \int_M (S-\rho)^2\, d\Vol_g$. By \eqref{eq:YF} we also conclude that
\[\int_M  \left( \partial_t u^{\frac{2n}{n-2}}\right) \, d\mu =-\frac{n}{2} \int_M (S-\rho)u^{\frac{2n}{n-2}} \, d\mu=0,\]
and using $u$ as a test function in \eqref{eq:YamabeFlow} leads to
\[\frac{n+2}{2n}\int_M \partial_t u^{\frac{2n}{n-2}}\, d\mu +(n-1)\int_M \vert \nabla u\vert^2 \, d\mu =\frac{n+2}{4}\left (\rho(t) -\int_M u^2S_0\, d\mu \right),\]
so
\[\int_M \vert \nabla u\vert^2 \, d\mu\leq  \frac{n+2}{4} \left(\rho(0)+\norm{(S_0)_-}_{L^{\infty}(M)}\right),\]
where we have used $\int_M u^{\frac{2n}{n-2}}\, d\mu=1$. This shows that $u$ is uniformly bounded in $H^1(M)$, independent of $t$ for all $t\geq 0$. Since the Sobolev embedding $H^1(M)\hookrightarrow L^q(M)$ is compact for $q<\frac{2n}{n-2}$ (see \cite[Proposition 1.6]{ACM}), we in particular get that $u$ has a convergent subsequence in $L^2(M)$ as $t\to \infty$, and we call this limit $u_{\infty}$.
\end{proof}
\begin{Rem}
The above methods would also show that $\partial_t u^{\frac{n+2}{n-2}}\to 0$ in $L^1(M)$, since we may use \eqref{eq:YF} and the Hölder inequality to write
\begin{align*}
\norm{\partial_t u^{\frac{n+2}{n-2}}}_{L^1(M)} &\leq \frac{n+2}{4} \norm{(S-\rho) u^{\frac{n}{n-2}}}_{L^2(M)}\norm{u^{\frac{2}{n-2}}}_{L^2(M)}\\
&\leq \frac{n+2}{4} \norm{(S-\rho) u^{\frac{n}{n-2}}}_{L^2(M)}.
\end{align*}
We then use the first part of the corollary to show that the right hand side tends to $0$.
\end{Rem}

\section{Future research directions and open problems}\label{open-section}

Long time existence alone does \textit{not} guarantee regularity of the limit solution $u_\infty \in L^2(M)$. Indeed, this has to be obstructed for the following two reasons. In the case of closed manifolds, we know that the Yamabe problem is not uniquely solvable on a round sphere, but so far we have not assumed that $(M,g_0)$ is not a sphere. In the singular setting, the Yamabe problem doesn't always have a solution, as demonstrated by Viaclovsky \cite{Orbifold}. We suspect that demanding 
$$
Y(M,g_0)<\lim\limits_{R\to 0} Y(B_R(p),g_0),
$$ 
for all $p\in \overline{M}$ is the required condition in our setting. Under this assumption, Akutagawa, Carron, and Mazzeo \cite{ACM} are able to solve the Yamabe problem for smoothly stratified spaces. For closed manifolds, this condition becomes $Y(M,g_0)<Y(\mathbb{S}^n,g_{\mathbb{S}^n})$, with the round metric $g_{\mathbb{S}^n}$, 
which is the assumption used by Brendle  \cite{Brendle} in his study of the Yamabe flow. Brendle's proof
of convergence of the Yamabe flow relies on the positive mass theorem, which is not available in the singular setting.

\end{document}